\documentclass[final,onefignum,onetabnum]{siamart250211}

\usepackage{amsfonts,amsmath, amssymb, mathrsfs}
\usepackage{graphicx}
\usepackage{epstopdf}
\usepackage{algorithmic}
\usepackage[T1]{fontenc} % to make accended words (such as Fr\'echet) searchable in the PDF
\usepackage{lmodern} % to use Latin Modern instead of Computer Modern font, for smoother text rendering (to be checked)

\Crefname{ALC@unique}{Line}{Lines}
\ifpdf
  \DeclareGraphicsExtensions{.eps,.pdf,.png,.jpg}
\else
  \DeclareGraphicsExtensions{.eps}
\fi

\usepackage{empheq}
\usepackage{cases}
\usepackage{adjustbox} 
\usepackage{subcaption}
\usepackage{amsopn}
\usepackage{enumitem}
\usepackage{marginnote}

\newsiamremark{assumption}{Assumption}
\newsiamremark{remark}{Remark}
\newsiamremark{hypothesis}{Hypothesis}
\newsiamthm{claim}{Claim}
\newsiamthm{example}{Example}
\crefname{definition}{Definition}{Definitions}
\Crefname{definition}{Definition}{Definitions}
\crefname{assumption}{Assumption}{Assumptions}
\crefname{hypothesis}{Hypothesis}{Hypotheses}
\crefname{remark}{Remark}{Remarks}
\Crefname{remark}{Remark}{Remarks}
\crefname{example}{Example}{Examples}

\usepackage{standalone}
\usepackage{tikz}
\usepackage{pgfplots}
\usetikzlibrary{shapes.geometric, calc, intersections, arrows.meta, backgrounds, angles, quotes}
\usepackage{xcolor}

\definecolor{darkgreen}{rgb}{0,0.6,0.1}
\definecolor{darkblue}{rgb}{0,0,0.7}
\definecolor{mygreen}{rgb}{0.3,0.8,0}
\definecolor{myblue}{rgb}{0,0.5,1}
\definecolor{myred}{rgb}{1,0,0}
\definecolor{myorange}{rgb}{1,0.7,0.2}
\definecolor{mypink}{rgb}{0.8,0.2,0.6}
\newcommand{\N}{\mathrm{N}}
\newcommand{\R}{\mathbb{R}}
\newcommand{\T}{\mathrm{T}}

\newcommand{\subjectto}{\mathrm{s.\,t.}}

\newcommand{\St}{\mathrm{St}}
\newcommand{\StXTX}{\mathrm{St}_{X^\top X}}

\newcommand{\ip}[2]{\left\langle #1,#2\right\rangle}
\newcommand{\tr}{\mathrm{tr}}
\newcommand{\norm}[1]{\left\lVert #1\right\rVert}
\newcommand{\normF}[1]{\left\lVert #1\right\rVert_{\mathrm{F}}}

\newcommand{\frob}{\mathrm{F}}

\DeclareMathOperator{\grad}{grad}
\newcommand{\myskew}{\mathrm{skew}}
\DeclareMathOperator{\Hess}{Hess}
\DeclareMathOperator{\sym}{sym}

\DeclareMathOperator{\D}{D}

\DeclareMathOperator{\Id}{Id}

\headers{A second-order method on the Stiefel manifold via Newton--Schulz}{X. Xiong, B. Gao, and P.-A. Absil}

\title{A second-order method landing on the Stiefel manifold via Newton--Schulz iteration\thanks{Submitted to the editors DATE.
\funding{This work was supported by the National Key R\&D Program of China (grant 2023YFA1009300). BG was supported by the National Natural Science Foundation of China (grant No.~12288201). PA was supported by the Fonds de la Recherche Scientifique -- FNRS under Grant no T.0001.23.}}
}

\author{
Xinhui Xiong\thanks{State Key Laboratory of Mathematical Sciences, Academy of Mathematics and Systems Science, Chinese  Academy of Sciences, and University of Chinese Academy of Sciences, China (xiongxinhui@amss.ac.cn).}
\and Bin Gao\thanks{State Key Laboratory of Mathematical Sciences, Academy of Mathematics and Systems Science, Chinese  Academy of Sciences, China (gaobin@lsec.cc.ac.cn).}
\and P.-A. Absil\thanks{ICTEAM Institute, UCLouvain, Belgium
(pa.absil@uclouvain.be).}
}

\begin{document}

\maketitle

% REQUIRED
\begin{abstract}
Retraction-free approaches offer attractive low-cost alternatives to Riemannian methods on the Stiefel manifold, but they are often first-order, which may limit the efficiency under high-accuracy requirements. To this end, we propose a second-order method landing on the Stiefel manifold without invoking retractions, which is proved to enjoy local quadratic (or superlinear for its inexact variant) convergence. The update consists of the sum of (i) a component tangent to the level set of the constraint-defining function that aims to reduce the objective and (ii) a component normal to the same level set that reduces the infeasibility. Specifically, we construct the normal component via Newton--Schulz, a~fixed-point iteration for orthogonalization. Moreover, we establish a geometric connection between the Newton--Schulz iteration and Stiefel manifolds, in which Newton--Schulz moves along the normal space. For the tangent component, we formulate a modified Newton equation that incorporates Newton--Schulz. Numerical experiments on the orthogonal Procrustes problem, principal component analysis, and real-data independent component analysis illustrate that the proposed method performs better than the existing methods. 
\end{abstract}

% REQUIRED

% REQUIRED

\section{Introduction}
Consider the optimization problem
\begin{equation}
\label{eq:problem}
    \begin{aligned}
        \min\ &\ \ \ f(X) \\
        \subjectto\ &\quad X\in\St(p,n):=\{X\in\mathbb{R}^{n\times p}\mid X^\top X=I_p\},        
    \end{aligned}
\end{equation}
where $I_p$ is the $p\times p$ identity matrix, $p\leq n$, $\St(p,n)$ is the \textit{Stiefel manifold}~\cite{edelman1998geometry}, and the objective function $f$ is twice continuously differentiable. Problem~\eqref{eq:problem} arises in numerous applications, including principal component analysis (PCA)~\cite{hotelling1933analysis}, canonical correlation analysis (CCA)~\cite{hotelling1936relations}, independent component analysis (ICA)~\cite{HyvarinenICA2000}, spectral clustering~\cite{ngJordanWeissSpectralClustering2001},
electronic structure calculations~\cite{gaoOrthogonalizationfreeParallelizableFramework2020}, and the orthogonal Procrustes problem~\cite{schonemannOrthogonalProcrustes1966}. More recently, orthogonality constraints and regularization have also been used in deep learning to improve training stability and robustness~\cite{cisseParsevalNetworks2017,bansalOrthogonalityRegularizationNeurIPS2018}. 

\subsection{Related work and motivation}
\label{subsec:Related_work}
Problem~\eqref{eq:problem} can be tackled by Riemannian methods~\cite{absilOptimizationAlgorithmsMatrix2008a,huBriefIntroductionManifold2020,boumalIntroductionOptimizationSmooth2023b} that preserve feasibility by retractions. Typical first-order methods, including Riemannian gradient descent~\cite{absilOptimizationAlgorithmsMatrix2008a}, Riemannian conjugate gradient~\cite{Sato2021}, and the Cayley-transform-based feasible method~\cite{wenFeasibleMethodOptimizationOrthogonality2013}, have local linear convergence. Alternatively, second-order methods, e.g., Riemannian trust-region methods~\cite{absilTrustRegionMethodsRiemannian2007} and Riemannian Newton-type methods~\cite{edelman1998geometry,absilOptimizationAlgorithmsMatrix2008a}, can achieve local superlinear or quadratic convergence by introducing a subproblem or a Newton equation. However, these methods can incur substantial per-iteration costs from retractions; see, e.g.,~\cite{gaoParallelizableAlgorithmsOptimization2019}. Specifically, depending on the chosen retraction, each iteration may require expensive linear algebra operations. This overhead can become significant in high-dimensional problems or when gradient evaluation is relatively cheap.

To avoid the cost of retractions, recent works studied infeasible methods that do not enforce the constraint at every iteration. Exact penalty models reformulate the constrained problem~\eqref{eq:problem} into an unconstrained one by employing different penalty functions, e.g., the modified augmented Lagrangian~\cite{gaoParallelizableAlgorithmsOptimization2019,xiaoLiuYuanSmoothExactPenalty2022}, Fletcher’s augmented Lagrangian~\cite{fletcherExactPenalty1973,goyensComputingSecondorderPoints2024}, the approximate augmented Lagrangian~\cite{wangDecentralizedOptimizationStiefel2022}, and the constraint-dissolving function~\cite{xiaoLiuExPen2024,xiaoLiuTohCDF2024}. 
More recently, alternating schemes were proposed to tackle~\eqref{eq:problem} by alternately searching along a gradient-related step and a constraint-related step. For instance, the sequential linearized proximal gradient method~\cite{liuXiaoYuanSLPG2024} solved a proximal linearized subproblem followed by a correction step approximating the polar decomposition. This method was further extended by~\cite{javaloyEmbarrassinglySimpleWay2026} in which the subproblem is substituted with a gradient step tangent to the level set of the constraint-defining function. Gratton and Toint~\cite{grattonSimpleFirstorderAlgorithm2025} adaptively switched between the gradient steps projected to the nullspace of the Jacobian of the constraints and an infeasibility-reduction step.

In contrast to alternating schemes, the \textit{landing} method, introduced by Ablin and Peyr\'e~\cite{ablinFastAccurateOptimization2022} for the orthogonal group and extended by Ablin et al.~\cite{ablinInfeasibleDeterministicStochastic2024} to the Stiefel manifold, avoids retractions by coupling an objective-improving direction with a correction direction that drives the iterates toward the manifold. The update takes the form:
\begin{equation}
\label{eq:FOL_intro_iteration}
    X_{k+1} = X_k - \eta_k\,\Lambda_1(X_k) \quad \text{with}\quad         \Lambda_1(X) := T_1(X) + \lambda\nabla\mathcal{N}(X), 
\end{equation}
 where $\Lambda_1$ is called the \textit{landing field}, $T_1(X):=2\myskew(\nabla f(X)X^\top)X$, $\nabla \mathcal{N}(X) = X(X^\top X-I_p)$ is the Euclidean gradient of \emph{infeasibility} measure
\begin{equation}
\label{eq:penalty_function}
    \mathcal N(X):=\frac{1}{4}\normF{X^\top X-I_p}^2,
\end{equation}
$\eta_k$ is the step size, and $\lambda >0$ is a hyperparameter. Gao et al.~\cite{gaoOptimizationFlowsLanding2022} provided a geometric interpretation of the landing field, in which $T_1$ and $\nabla\mathcal{N}$ are respectively tangent and normal to the current layered manifold~\cite{goyensComputingSecondorderPoints2024,goyensGeometricDesignTangent2025f}
\begin{equation}
\label{eq:layered_manifold}
    \mathrm{St}_{X_k^\top X^{}_k}(p,n) := \{Z \in \R^{n\times p} : Z^\top Z = X_k^\top X^{}_k\}.
\end{equation} 
Each update of the landing method involves only matrix multiplications, which is retraction-free and favorable for efficient GPU implementation. Recent works extended this framework to stochastic settings~\cite{ablinInfeasibleDeterministicStochastic2024,varyOptimizationRetractionRandom2024} and distributed optimization~\cite{sunRetractionfreeDecentralizedNonconvex2024,songDistributedRetractionfreeCommunicationefficient2025}. In addition, the local linear convergence of the landing method was proved in~\cite{sunLocalLinearConvergence2024} with a local Riemannian Polyak--{\L}ojasiewicz condition.
Beyond orthogonality constraints, Schechtman et al.~\cite{schechtmanOrthogonalDirectionsConstrained2023} and Vary et al.~\cite{varyOptimizationRetractionRandom2024} generalized the landing framework to equality-constrained optimization, which was further extended to inequality constraints~\cite{shiAdaptiveDirectionalDecomposition2025}. More recently, Goyens and Feppon~\cite{goyensRiemannianLandingMethod} proposed a Riemannian landing method based on a carefully designed metric with local quadratic convergence, and the method was proved to be equivalent to sequential quadratic programming~(SQP). Si and Malick~\cite{SiMalickLanding2026}, as well as Goyens and Feppon~\cite{goyensRiemannianLandingMethod}, integrated backtracking line search into the landing method.

Several existing methods employ a normal update that yields quadratic reduction of infeasibility~\cite[Lem.~11]{liuXiaoYuanSLPG2024}, but they do not exploit second-order information on the objective function, and hence they fail to achieve local superlinear convergence. This motivates the development of a method that makes explicit use of such a normal update while achieving local superlinear convergence.

\subsection{Contributions}
In this paper, we propose a second-order landing (SOL) framework to solve optimization problems on the Stiefel manifold~\eqref{eq:problem}.
At each iteration, the landing update
\begin{subequations}  \label{eq:second_order_landing_framework}
\begin{equation}
X_{k+1} = X_k + \eta_k \Lambda(X_k)
\end{equation}
is retraction-free and driven by a field $\Lambda$ with two components:
\begin{equation}
\Lambda(X) = T(X)+N(X),
\end{equation}
\end{subequations}
where $N(X)$, normal to the layered manifold~\eqref{eq:layered_manifold}, is able to reduce the infeasibility quadratically. The component $T(X)$, tangent to the layered manifold, incorporates second-order information to minimize the objective, enabling local quadratic convergence.

To construct the normal component $N(X)$, we leverage the Newton--Schulz (NS) iteration~\cite{schulzIterativeBerechungReziproken1933}, an iterative method for matrix orthogonalization~\cite{kovarikIterativeMethodsImproving1970,bjorckIterativeAlgorithmComputing1971}. We prove that the Newton--Schulz iteration, $Y_{k+1} = Y_k + N_r(Y_k)$, moves along the normal space of the layered manifold $\mathrm{St}_{Y_k^\top Y^{}_k}(p,n)$, i.e., an order-$r$ NS update
\begin{equation}
    \label{eq:NS_direction_intro}
    N_r(X) := X\left(\sum_{j=0}^{r}(-1)^j\frac{(2j)!}{(j!)^2}\frac{(X^\top X-I_p)^j}{4^j}-I_p\right)
\end{equation}
belongs to the normal space; see~\cref{subsec:normal_geometry}. Hence, we establish a geometric interpretation between the Newton--Schulz iteration and Stiefel manifolds. Moreover, we show that it yields local infeasibility decay of order $r+1$. It is worth noting that the order-$1$ NS update reduces to 
\[
N_1(X)=- \frac 1 2 X(X^\top X-I_p)=-\frac12 \nabla\mathcal{N}(X)
\]
which yields a quadratic reduction in infeasibility, and thus provides a reasonable normal update without any hyperparameters. We adopt $N(X)=N_1(X)$ as the normal component in~\eqref{eq:second_order_landing_framework}. In fact, we notice that the infeasibility-reduction step used in~\cite{liuXiaoYuanSLPG2024,javaloyEmbarrassinglySimpleWay2026,peng2026ns} is exactly one step of the order-$1$ NS iteration. These works employ it purely within first-order frameworks, whereas we provide a geometric view of the update. Furthermore, we exploit its quadratic infeasibility contraction to construct a local quadratically convergent method.

For the tangent component $T(X)$, an intuitive choice is the Riemannian Newton~(RN) update $T_{\mathrm{RN}}$ on the layered manifold, obtained by solving the RN equation. However, a simple combination of the order-$1$ NS update $N_1(X)$ and $T_{\mathrm{RN}}$ does not attain the sought quadratic convergence; see~\cref{fig:RN_vs_N} for an illustrative example. A possible reason is that the RN equation is solely a second-order approximation to the objective on the current layered manifold, while ignoring movement in the normal space. Therefore, the tangent component $T(X)$ should be carefully developed and incorporate the normal component. To this end, we construct a corrected tangent update $T(X)$ by solving a modified Newton equation 
\[
\Hess f(X)[T(X)] = -\grad f(X)-\mathcal A_{\mathrm{N}}(X)[N(X)]
\]
such that the sequence of second-order landing enjoys local quadratic convergence, where $\mathcal{A}_{\mathrm{N}}(X)$, which anticipates the normal component, is a correction operator onto the tangent space; see~\cref{subsec:symmetric_tangent_system} for details. Furthermore, we propose an approximate Newton equation to reduce the computational cost; see~\cref{subsec:corrected_tan_direction}.

\begin{figure}[tbp]
    \centering
     \begin{minipage}[t]{0.58\linewidth}
       \centering
        \vspace{0pt}
        \resizebox{\linewidth}{!}{%
\begin{tikzpicture}[
    % --- Global settings ---
    >=stealth, 
    font=\footnotesize, 
    manifold/.style={line width=1.2pt, color=black},
    helper/.style={dash pattern=on 3pt off 2pt, line width=0.6pt, color=gray!80},
    vec/.style={->, line width=1.0pt}, 
    labelbg/.style={inner sep=1.5pt, outer sep=0pt},
    scale=0.99,
    transform shape,
]

    % =========================================================
    % 0. color
    % =========================================================
    \definecolor{mainRed}{RGB}{214, 39, 40}    
    \definecolor{mainGreen}{RGB}{44, 160, 44}  
    \definecolor{mainBlue}{RGB}{31, 119, 180}  
    \definecolor{darkBlue}{RGB}{120,75,160}
    \definecolor{bgGray}{RGB}{248, 248, 248} 

    % =========================================================
    % 1. Parameter
    % =========================================================
    \def\Rmain{16}    
    \def\Rout{19}     
    \def\Rbottom{15} 
    \def\RX{2.0}   
    \def\angStart{103} 
    \def\angEnd{77}  
    \def\angKKT{87} 
    \def\angX{\angKKT+10} 
    \def\angTrajEnd{\angKKT-5} 

    \coordinate (X) at (\angX:\Rmain+\RX);
    
    \pgfmathsetmacro{\vecMyGreenLen}{0.85*\RX}    
    \pgfmathsetmacro{\angMyGreen}{\angX + 180}
    \pgfmathsetmacro{\vecDarkBlueLen}{4.2}       
    \pgfmathsetmacro{\angDarkBlue}{\angX - 90}
    \pgfmathsetmacro{\vecMyBlueLen}{3}        
    \pgfmathsetmacro{\angMyBlue}{\angX - 90}
    
    % =========================================================
    % 2. Background
    % =========================================================
    \begin{scope}[on background layer]
        \fill[gray!10] (\angStart:\Rout) arc (\angStart:\angEnd:\Rout) --
              (\angEnd:\Rbottom) -- (\angStart:\Rbottom) -- cycle;
    \end{scope}

    % =========================================================
    % 3. Manifold
    % =========================================================
    \draw[helper, gray!40] (\angStart:\Rout) arc (\angStart:\angEnd:\Rout);
    \draw[helper, name path = arcX] (\angX+3:\Rmain+\RX) arc (\angX+3:\angTrajEnd-1:\Rmain+\RX);
    \draw[helper, name path = arcXhat] (\angX+3:\Rmain+\RX-\vecMyGreenLen) arc (\angX+3:\angTrajEnd-1:\Rmain+\RX-\vecMyGreenLen);
    
    \draw[manifold] (\angStart:\Rmain) 
        node[below left, black, xshift=7pt, font=\small\bfseries] {$\mathrm{St}(p, n)$} 
        arc (\angStart:\angEnd:\Rmain);
    
    % =========================================================
    % 4. Trajectory
    % =========================================================
    \coordinate (BottomLeft) at (\angStart:\Rbottom);
    \coordinate (BottomRight) at (\angEnd:\Rbottom);    
    \coordinate (Xstar) at (\angKKT:\Rmain);
    \coordinate (TrajStart) at ($(BottomLeft)!0.55!(BottomRight)$);
    \coordinate (TrajEnd) at (\angTrajEnd:\Rout); 
    
    \draw[mainRed, dashed, line width=1.5pt, name path = traj] plot [smooth, tension=1] coordinates {
        (TrajStart) (Xstar) (TrajEnd)
    };

    \path [name intersections={of=traj and arcX, by=IntPoint1}];
    \path [name intersections={of=traj and arcXhat, by=IntPoint2}];

    \foreach \pt in {IntPoint1} {
        \node[star, star points=5, star point ratio=2.25, fill=darkBlue, draw=white, line width=0.5pt, inner sep=0pt, minimum size=8pt] at (\pt) {};
    }

    \foreach \pt in {IntPoint2} {
        \node[star, star points=5, star point ratio=2.25, fill=mainBlue, draw=white, line width=0.5pt, inner sep=0pt, minimum size=8pt] at (\pt) {};
    }

    \foreach \pt in {Xstar} {
        \node[star, star points=5, star point ratio=2.25, fill=mainRed, draw=white, line width=0.5pt, inner sep=0pt, minimum size=8pt] at (\pt) {};
    }

    \node[labelbg, anchor=north west, xshift=-1pt, yshift=-2pt, text=mainRed] at (Xstar) {$X_\star$};

\node[mainRed, right, align=left, font=\scriptsize, yshift=15pt, xshift=-77pt] 
          at (TrajStart) 
          { \fontsize{7pt}{6pt}\selectfont \text{Optimal solutions on} \\ \ \ \ \text{layered manifolds}};
    % =========================================================
    % 5. Vectors
    % =========================================================
    \coordinate (TipDarkBlue) at ($(X) + (\angDarkBlue:\vecDarkBlueLen)$);
    \coordinate (TipMyGreen) at ($(X) + (\angMyGreen:\vecMyGreenLen)$);
    \coordinate (TipMyBlue) at ($(X) + (\angMyBlue:\vecMyBlueLen)$);

    \coordinate (TipBlack1) at ($(TipDarkBlue) + (TipMyGreen) - (X)$);
    \coordinate (TipBlack2) at ($(TipMyBlue) + (TipMyGreen) - (X)$);

\begin{scope}[on background layer]
  \draw pic[draw, angle radius=5pt, angle eccentricity=1.2]
    {right angle = TipMyGreen--X--TipMyBlue};
\end{scope}
    
    \draw[densely dotted, gray, line width=0.6pt] (TipDarkBlue) -- (TipBlack1);
    \draw[densely dotted, gray, line width=0.6pt] (TipMyGreen) -- (TipBlack1);
    
    \draw[densely dotted, gray, line width=0.6pt] (TipMyBlue) -- (TipBlack2);

    \draw[vec, mainGreen] (X) -- (TipMyGreen) 
         node[midway, left, labelbg, text=mainGreen, xshift=-2pt]
         {$N_1(X)$};

    \draw[vec, darkBlue, line width=0.8pt, dash dot] (X) -- (TipDarkBlue) 
         node[above left, labelbg, text=darkBlue, xshift=1pt] {$T_{\mathrm{RN}}(X)$};

    \draw[->, >=stealth, line width=1.0pt, color=black!80, dash dot] (X) -- (TipBlack1) 
        node[right, labelbg, xshift=2pt] {$\Lambda_{\mathrm{RN}}(X)$};
    
    \draw[vec, mainBlue, line width=1.2pt] (X) -- (TipMyBlue) 
         node[above left, labelbg, text=mainBlue, xshift=-3pt, yshift=1pt] {$T(X)$};
    
    \draw[vec, black, line width=1.2pt] (X) -- (TipBlack2) 
        node[midway, below left, labelbg, xshift=1pt, yshift=1pt] {$\Lambda(X)$};

    % =========================================================
    % 6. Point
    % =========================================================
    \fill[black] (X) circle (1.8pt) node[above, labelbg, yshift=2pt] {$X$};

\end{tikzpicture}%
}
    \end{minipage}
    \hfill
    \begin{minipage}[t]{0.41\linewidth}
        \centering
        \vspace{10pt}
        \includegraphics[width=\linewidth]{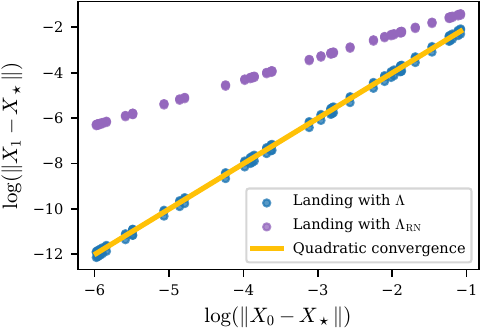}
    \end{minipage}

\caption{An illustration of second-order landing methods with intuitive and corrected Riemannian Newton updates, $\Lambda_{\mathrm{RN}}=T_{\mathrm{RN}}+N_1$ and $\Lambda=T+N_1$. Left: A geometric illustration shows that $T_{\mathrm{RN}}$ aims to find the solution on the current layer while the corrected update $T$ brings $\Lambda$ to the solution on the next layer; Right: A numerical comparison on one step of second-order landing methods with different initial points reports that in a small asymptotic regime, landing with $\Lambda$ enjoys quadratic convergence (illustrated by the yellow reference line of slope 2), while an intuitive landing with $\Lambda_{\mathrm{RN}}$ does not.
}
\label{fig:RN_vs_N}
\end{figure}
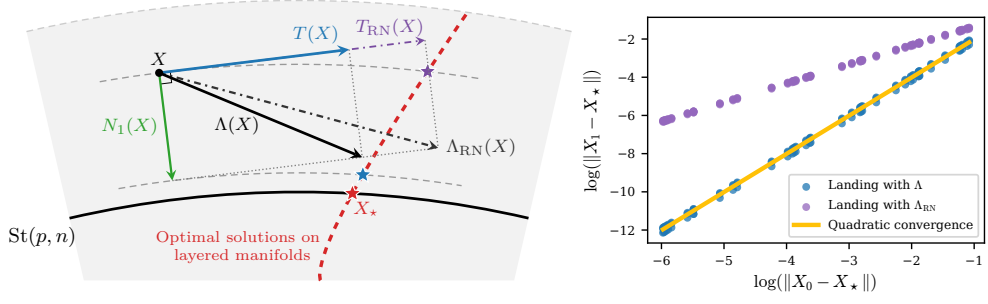

We prove that the SOL method achieves local quadratic convergence, and its inexact variant attains a local superlinear convergence. Numerical experiments on the orthogonal Procrustes problem, principal component analysis, and real-data independent component analysis validate the fast local convergence of the proposed SOL method, and demonstrate that SOL performs better than the existing penalty-based method, the inexact Riemannian Newton method, and SQP.

\subsection{Notation}
Let $\mathcal S_{\sym}(m)$ and $\mathcal{S}_{\myskew}(m)$ be the sets of all $m\times m$ real symmetric and skew-symmetric matrices, respectively. The set of all $n\times p$ real matrices with full column rank is denoted by $\mathbb{R}^{n\times p}_*$. Let $\|\cdot\|$ denote the operator norm for linear operators, $\|\cdot\|_2$ the spectral norm for matrices, and $\|\cdot\|_{\mathrm F}$ the Frobenius norm.
The $i$-th largest singular value of a matrix $M$ is denoted by $\sigma_{i}(M)$, and the smallest one by $\sigma_{\min}(M)$.
The Euclidean inner product is given by $\ip{X_1}{X_2}:=\tr(X_1^\top X_2)$ with trace operator $\tr(\,\cdot\,)$.
For a square matrix $A$, its symmetric and skew-symmetric parts are defined by $\sym(A):=\frac{1}{2}(A+A^\top)$ and $\myskew(A):=\frac12(A-A^\top)$, respectively. The Fr{\'e}chet derivative of a mapping $F:\mathcal{E}_1\to\mathcal{E}_2$ at $X$ is denoted by $\mathrm{D}F(X)$.
For a smooth function $f$ defined on a Euclidean space, 
$\nabla f$ and $\nabla^2 f$ denote its Euclidean gradient and Euclidean Hessian, respectively.
Given $X\in\mathbb{R}^{n\times p}_*$, $\St_{X^\top X}(p,n)$ denotes the layered manifold, as defined in~\eqref{eq:layered_manifold}. Throughout this paper, unless otherwise stated, $\grad f(X)$ and $\Hess f(X)$ denote the Riemannian gradient and Riemannian Hessian of $f$ on $\St_{X^\top X}(p,n)$, respectively, with respect to the Riemannian metric $g$ defined in~\eqref{eq:extended_canonical_metric}.

\subsection{Organization}
\Cref{section:Pre} reviews the first-order landing algorithm alongside its geometric interpretation.
We establish the geometric interpretation and infeasibility-reduction property of the Newton--Schulz iteration in~\cref{sec:design_normal_term}, and construct the normal component. \Cref{sec:design_tangent_term} identifies the limitation of the Riemannian Newton update and derives a modified Newton equation, together with its tractable approximation.
We propose the SOL method and its inexact variant in~\cref{sec:SOL_methods}, and prove the local convergences. Numerical experiments are provided in \cref{sec:numerical_examples} and concluding remarks are given in \cref{sec:conclusion}.

\section{The landing method}
\label{section:Pre}

In this section, we review the preliminaries of the landing framework~\cite{ablinFastAccurateOptimization2022,ablinInfeasibleDeterministicStochastic2024} and present its geometric interpretation~\cite{gaoOptimizationFlowsLanding2022}. 

The first-order landing method~\cite{ablinFastAccurateOptimization2022,ablinInfeasibleDeterministicStochastic2024} tackles~\eqref{eq:problem} by an update as in~\eqref{eq:FOL_intro_iteration}:
\begin{equation}
\label{eq:FOL_iteration}
\begin{aligned}
        \Lambda_1(X) &= T_1(X) + \lambda\nabla\mathcal{N}(X), 
        \\
        X_{k+1} &= X_k - \eta_k\,\Lambda_1(X_k),
\end{aligned}
\end{equation}
where 
 \begin{equation}
 \label{eq:relative_grad}
T_1(X)=2\,\myskew(\nabla f(X)X^\top)X,
 \end{equation}
$\nabla \mathcal{N}(X) = X(X^\top X-I_p)$ is the Euclidean gradient of $\mathcal{N}(X)$ in~\eqref{eq:penalty_function}, and $\lambda >0$ is a hyperparameter.
This additive structure admits a geometric interpretation in terms of the layered manifold $\St_{X^\top X^{}}(p,n)$~\eqref{eq:layered_manifold}: $T_1(X)$ belongs to the tangent space and $\nabla\mathcal{N}(X)$ belongs to the normal space; see~\cite{gaoOptimizationFlowsLanding2022}.

Let $X\in\mathbb{R}^{n\times p}_*$. The tangent space at $Y\in\St_{X^\top X^{}}(p,n)$ can be parametrized as
\begin{equation}
\label{eq:tangent_space}
\begin{aligned}
\mathrm T_Y\mathrm{St}_{X^\top X}(p,n)
&=\{\xi\in\mathbb{R}^{n\times p}:\xi^\top Y+Y^\top\xi=0\}\\
&=\{W Y:W\in\mathcal{S}_{\mathrm{skew}}(n)\}.
\end{aligned}
\end{equation}
Let the ambient space $\R_*^{n\times p}$ be equipped with the following metric for all $Y \in \R_*^{n\times p}$ and $\xi, \eta \in \R^{n\times p}$,
\begin{equation}
\label{eq:extended_canonical_metric}
g_Y(\xi,\eta):=\ip{\xi}{(I_p-\frac12Y(Y^\top Y)^{-1}Y^\top)\eta(Y^\top Y)^{-1}}.
\end{equation}
When the layered manifold $\St_{X^\top X}(p,n)$ is viewed as a Riemannian submanifold of $(\mathbb{R}^{n\times p}_*, g)$, the component $T_1(X)$ coincides with the Riemannian gradient of $f$ restricted to $\St_{X^\top X}(p,n)$, i.e., 
\begin{equation}
    \label{eq:Riemannian_gradient}
    \grad f(X)=T_1(X),
\end{equation}
where the Riemannian gradient~\cite[Eq.\;(3.31)]{absilOptimizationAlgorithmsMatrix2008a} is the unique tangent vector satisfying $g_X(\grad f(X),\eta)=\mathrm{D} f(X)[\eta]$ for all $\eta\in\mathrm{T}_X\mathrm{St}_{X^\top X}(p,n)$. The second term in $\Lambda_1(X)$ involves the Euclidean gradient of the infeasibility measure $\mathcal{N}(X)$ defined in~\eqref{eq:penalty_function}. The normal space of $\mathrm{St}_{X^\top X}(p,n)$ under the metric $g$ is
\begin{equation}
\label{eq:normal_space}
\N_Y\mathrm{St}_{X^\top X}(p,n)
=\{Y(X^\top X)^{-1}S:S \in \mathcal{S}_{\mathrm{sym}}(p)\}.
\end{equation}
Moreover, we have $\nabla \mathcal N(X)=X(X^\top X)^{-1}((X^\top X)^2-X^\top X)\in \N_X(\mathrm{St}_{X^\top X}(p,n))$. 

In addition, the two components are orthogonal under the Euclidean metric, $\langle T_1(X),\, \nabla\mathcal{N}(X) \rangle = 2\langle \mathrm{skew}(\nabla f(X)X^\top),\, X(X^\top X - I_p)X^\top \rangle = 0$. 
Consequently, for all $X \in \R_*^{n\times p}$, the landing field $\Lambda_1(X)$ vanishes if and only if both $T_1(X)=0$ and $\nabla\mathcal{N}(X)=0$, i.e.,
\[
 \grad f(X)=0 \quad \text{and} \quad X\in\St(p,n),
\]
which characterizes the first-order stationary points~\cite{absilOptimizationAlgorithmsMatrix2008a} of the original constrained problem~\eqref{eq:problem}.

Although the iterates are not constrained on the Stiefel manifold, the convergence analysis relies on a local \emph{safe region}.
\begin{definition}
\label{def:safety_region}
    The safe region with $\varepsilon \in (0,1)$ is defined as $\mathrm{St}(p,n)^\varepsilon:=\{X\in\mathbb{R}^{n\times p}\mid\|X^\top X-I_p\|_\mathrm{F}\leq\varepsilon\}$.
\end{definition}

The first-order landing method has several appealing features: it avoids retractions, involves only matrix multiplications, and still enjoys an iteration complexity bound comparable to that of 
its Riemannian gradient counterpart (see, e.g.,~\cite[Prop.~9]{ablinInfeasibleDeterministicStochastic2024}). We utilize the framework of first-order landing and develop a retraction-free second-order method. The following sections are devoted to concretizing $T(X)$ and $N(X)$ such that the update~\eqref{eq:second_order_landing_framework} achieves fast local convergence.

\section{The normal component}
\label{sec:design_normal_term}
In this section, we construct the normal component $N(X)$ in the second-order landing framework~\eqref{eq:second_order_landing_framework}, in which the role of $N(X)$ is to reduce infeasibility. The recent renaissance of the Newton--Schulz (NS) iteration~\cite{schulzIterativeBerechungReziproken1933} in large language models~\cite{jordan2024muon} inspires the development of normal updates. Moreover, we provide a geometric interpretation of the NS iteration.

The Newton--Schulz iteration provides a cheap approximation to the polar factor $Y_0(Y_0^\top Y_0)^{-1/2}$ for a matrix $Y_0\in\mathbb{R}^{n\times p}_*$.
As noted by Kovarik~\cite{kovarikIterativeMethodsImproving1970} and Björck et al.~\cite{bjorckIterativeAlgorithmComputing1971}, the inverse square root can be approximated by a truncated Taylor expansion. Specifically, given the current iterate $Y_k$ and defining
\begin{equation}
    \label{eq:E_k_notation}
    E_k := Y_k^\top Y^{}_k - I_p,
\end{equation}
the order-$r$ Newton--Schulz iteration takes the form
\begin{equation}
    \label{eq:NS_multiplicative}
    Y_{k+1} = Y_k\, q_r(E_k),
    \qquad
    q_r(S) := \sum_{j=0}^{r}(-1)^j\frac{(2j)!}{(j!)^2}\frac{S^j}{4^j},
\end{equation}
where $q_r(S)$ is the degree-$r$ Taylor polynomial approximation of $(I_p+S)^{-1/2}$. 
Since this iteration requires only matrix multiplications, it is well-suited to GPU implementation and has recently been used for large-scale orthogonalization in machine learning, e.g., in the \textit{Muon} optimizer~\cite{jordan2024muon}. As shown in \cref{prop:NS_geometry}, the  iteration~\eqref{eq:NS_multiplicative} also admits a geometric interpretation as a displacement normal to the layered manifold, making it appropriate to construct the normal component.

\subsection{Geometric interpretation of the Newton--Schulz iteration}
\label{subsec:normal_geometry}

To fit the landing framework, which is formulated in additive form, we rewrite the Newton--Schulz iteration~\eqref{eq:NS_multiplicative} as defined in~\eqref{eq:NS_direction_intro},
\begin{equation}
    \label{eq:NS_additive}
    Y_{k+1} = Y_k + N_r(Y_k),
    \qquad
    N_r(Y_k) = Y_k\bigl(q_r(E_k)-I_p\bigr).
\end{equation}
Thus, the Newton--Schulz iteration can be interpreted as an additive update along $N_r(Y_k)$. The following proposition shows that $N_r(Y_k)$ belongs to the normal space of the layered Riemannian manifold $\St_{Y_k^\top Y^{}_k}(p,n)$ at $Y_k$ with respect to the metric $g$ defined in~\eqref{eq:extended_canonical_metric}. 

\begin{proposition}[geometric interpretation of Newton--Schulz]\label{prop:NS_geometry}
Let $r \ge 1$ be a fixed integer. For all $Y\in\mathbb{R}^{n \times p}_*$, the order-$r$ Newton--Schulz update $N_r(Y)$ in~\eqref{eq:NS_additive} lies in the normal space~\eqref{eq:normal_space} of the layered manifold at $Y$, i.e.,
\[
    N_r(Y) \in \mathrm{N}_{Y}\mathrm{St}_{Y^\top Y}(p,n).
\]
\end{proposition}

\begin{proof}
    It follows from~\eqref{eq:NS_additive} and~\eqref{eq:E_k_notation} that
    \[
    N_r(Y) = Y (q_r(Y^\top Y-I_p) - I_p) = Y(Y^\top Y)^{-1} \underbrace{(Y^\top Y) (q_r(Y^\top Y-I_p) - I_p)}_{S}.
    \]
    Since $Y^\top Y$ and $q_r(Y^\top Y-I_p)$ are both polynomials of the symmetric matrix $Y^\top Y$, they are symmetric and commute. Consequently, the product $S$ is symmetric. In view of the normal space~\eqref{eq:normal_space}, $N_r(Y) \in \mathrm{N}_{Y}\mathrm{St}_{Y^\top Y}(p,n)$.
\end{proof}

From the proof of~\cref{prop:NS_geometry}, we observe that the NS update $N_r(Y)$ also lies in the normal space of $\St_{Y^\top Y^{}}(p,n)$ under the standard Euclidean metric, which is $\{YS:S \in \mathcal{S}_{\mathrm{sym}}(p)\}$.

The efficiency of the NS iteration is rooted in its rapid convergence for orthogonalization. The next proposition quantifies the high-order contraction rate of the NS iteration; see a similar result for square matrices~\cite{KenneyRationalIterativeMatrixSign1991}.

\begin{proposition}
\label{prop:NS_convergence} 
Let $Y_0\in\St(p,n)^\varepsilon$ where \(\St(p,n)^\varepsilon\) is the safe region from~\cref{def:safety_region} with $\varepsilon\in(0,1)$. Let $(Y_k)$ be the sequence of the order-$r$ Newton--Schulz iteration~\eqref{eq:NS_multiplicative} with the initial point $Y_0$. It holds that 
\[
\normF{Y_k^\top Y^{}_k-I_p} \le \normF{Y_0^\top Y^{}_0-I_p}^{(r+1)^k}\leq\varepsilon^{(r+1)^k} .
\]
\end{proposition}

\begin{proof}
    Let $Y_0=U\Sigma_0V^\top$ be the thin singular value decomposition of $Y_0$ and $(\Sigma_k)$ be generated by the order-$r$ Newton--Schulz iteration~\eqref{eq:NS_multiplicative} with the initial point~$\Sigma_0$. Since the NS iteration is orthogonally invariant, it holds that $Y_k = U\Sigma_kV^\top$. By applying~\cite[Thm.~5.2]{KenneyRationalIterativeMatrixSign1991} to the square matrices $(\Sigma_k)$ and combining $Y_0\in\St(p,n)^\varepsilon$, we have $\normF{Y_k^\top Y^{}_k-I_p}=\normF{\Sigma_k^2-I_p}  \leq \normF{\Sigma^2_0-I_p}^{(r+1)^k}=\normF{Y_0^\top Y^{}_0-I_p}^{(r+1)^k}\leq\varepsilon^{(r+1)^k}$.
\end{proof}

In the next result, we record that choosing $r=1$ yields local quadratic convergence, a property already obtained in~\cite[Lem.~11]{liuXiaoYuanSLPG2024}.

\begin{corollary}
\label{cor:quadratic_decay_NS}
Let $Y_0\in\St(p,n)^\varepsilon$ where \(\St(p,n)^\varepsilon\) is the safe region from~\cref{def:safety_region} with $\varepsilon\in(0,1)$. Let $(Y_k)$ be generated by the order-$1$ Newton--Schulz iteration\begin{equation}\label{eq:NS_update_order_1}
    Y_{k+1}=Y_k+N_1(Y_k)=Y_k(\tfrac32 I_p-\tfrac12 Y_k^\top Y^{}_k),
\end{equation}
    with the initial point $Y_0$. It follows that
\[
\normF{Y_k^\top Y^{}_k-I_p} \le \normF{Y_0^\top Y^{}_0-I_p}^{2^k}\leq\varepsilon^{2^k} .
\]
\end{corollary}

The above results show that the Newton--Schulz iteration is favorable for the second-order landing framework~\eqref{eq:second_order_landing_framework} as it provides a geometrically valid normal update and reduces infeasibility at a high-order rate using only matrix multiplications.

\subsection{The Newton--Schulz normal update}
\label{subsec:normal_construction}

In view of~\cref{cor:quadratic_decay_NS}, we adopt the order-1 NS update in the proposed SOL framework~\eqref{eq:second_order_landing_framework}:
\begin{equation}
\label{eq:normal_component}
    N(X):=N_1(X)= -\tfrac12 X(X^\top X^{}-I_p)=-\tfrac12 \nabla\mathcal N(X).
\end{equation}
Notice that $N_1(X)$ coincides with the normal component in the first-order landing iteration~\eqref{eq:FOL_iteration} with $\lambda=\tfrac12$. Thus, the normal component is specified without introducing a tunable hyperparameter.

\begin{remark}
    As mentioned in the introduction, several existing methods (e.g.,~\cite{liuXiaoYuanSLPG2024,javaloyEmbarrassinglySimpleWay2026,peng2026ns}) adopt the same update as~\eqref{eq:NS_update_order_1} to reduce infeasibility. These works, however, did not fully make use of the quadratic infeasibility-reduction property (\cref{cor:quadratic_decay_NS}) to attain fast local convergence. In contrast, we not only interpret~\eqref{eq:NS_update_order_1} geometrically but also exploit its properties to develop a locally quadratically convergent method.
\end{remark}

\section{The tangent component}
\label{sec:design_tangent_term}
This section derives the tangent component for the second-order landing framework~\eqref{eq:second_order_landing_framework}. While the Riemannian Newton update on the associated layered manifold serves as an intuitive candidate, it fails to attain local quadratic convergence, see~\cref{fig:RN_vs_N}. To compensate for the normal-induced perturbation of the Riemannian gradient, we formulate a modified Newton equation to yield a valid correction. Furthermore, to enhance computational efficiency, we conclude by proposing an approximation to this equation.
 
\subsection{An intuitive choice for the tangent component}
\label{subsect:RN_tangent_direction}
In light of the geometric insights in~\cref{section:Pre}, it is intuitive to construct the tangent component based on the Riemannian Newton update on $\mathrm{St}_{X^\top X}(p,n)$. 
Evaluating this update requires the Riemannian Hessian~\cite[Def.~5.5.1]{absilOptimizationAlgorithmsMatrix2008a} of $f$
with respect to the metric~$g$ defined in~\eqref{eq:extended_canonical_metric}, which involves
the $g$-orthogonal projection $\Pi_X^{\mathrm{T}}(A)$ from $\mathbb{R}^{n\times p}$ onto $\T_X\mathrm{St}_{X^\top X}(p,n)$ with expression~\cite[Prop.~4]{goyensGeometricDesignTangent2025f}
\begin{equation}
\label{eq:projection_T}
\Pi_X^{\mathrm{T}}(A)=A-X(X^\top X)^{-1}\sym(X^\top A).
\end{equation}

\begin{proposition}[Riemannian Hessian]
\label{prop:Riemannian_Hessian_action}
Let $X\in\R^{n\times p}_*$. For all tangent vectors $V\in \T_X\mathrm{St}_{X^\top X}(p,n)$, the Riemannian Hessian of $f$ at $X$ is given by
\begin{equation}
\label{eq:Riemannian_Hessian}
\begin{aligned}
\Hess f(X)[V]
={}&\Pi_X^{\mathrm{T}}\Big(
2\,\myskew(\nabla^2 f(X)[V]X^\top)X
+2\,\myskew(\nabla f(X)V^\top)X \\
&\qquad\quad
+2\,\myskew(\nabla f(X)X^\top)V
-(I+P_X)\,\Xi_X(V)
\Big),
\end{aligned}
\end{equation}
where $Q_X:=(X^\top X)^{-1}$, $P_X:=XQ_XX^\top$, $G(X):=2\myskew(\nabla f(X)X^\top)X$,
\begin{equation}
\label{eq:Xi_term}
\Xi_X(V)
:=
\frac12\Big(VQ_XX^\top G(X)+G(X)Q_XX^\top V\Big)
+\frac14\,XQ_X\Big(V^\top G(X)+G(X)^\top V\Big),
\end{equation}
and $\Pi_X^{\mathrm{T}}$ is the $g$-orthogonal projection onto $\T_X\mathrm{St}_{X^\top X}(p,n)$ given in~\eqref{eq:projection_T}.
\end{proposition}

\begin{proof}
By the definition of the Riemannian Hessian~\cite[Def.~5.5.1]{absilOptimizationAlgorithmsMatrix2008a}, $\Hess f(X)[V]=\Pi_X^{\mathrm{T}}(\bar\nabla_V G)$, it suffices to compute the ambient Levi--Civita connection $\bar\nabla$ of the metric $g_X(U,W)=\ip{U}{\mathcal{M}_X[W]}$, where $\mathcal M_X[W]:=\Bigl(I-\frac12P_X\Bigr)WQ_X$. 

To derive the Levi--Civita connection, let $S_U := X^\top U + U^\top X$ for $U \in \R^{n\times p}$. Differentiating $Q_X$ and $P_X$ yields $\mathrm{D}Q_X[U] = -Q_XS_UQ_X$ and $\mathrm{D}P_X[U] = UQ_XX^\top + XQ_XU^\top - XQ_XS_UQ_XX^\top$. Therefore, 
\[\mathrm{D}\mathcal{M}_X[U]W = -\frac{1}{2}\mathrm{D}P_X[U]WQ_X - (I - \frac{1}{2}P_X)WQ_XS_UQ_X.\] For tangent vectors $U,Z,W \in T_X\mathrm{St}_{X^\top X}(p,n)$, the condition $S_U=S_Z=S_W=0$ simplifies this to $\mathrm{D}\mathcal{M}_X[U]W = -\frac{1}{2}(UQ_XX^\top + XQ_XU^\top)WQ_X$. Applying the Koszul formula~\cite[Eq.~5.11]{absilOptimizationAlgorithmsMatrix2008a} to constant ambient vector fields gives
\[
2g_X(\Gamma_X(U,Z),W) = \langle Z, \mathrm{D}\mathcal{M}_X[U]W \rangle + \langle U, \mathrm{D}\mathcal{M}_X[Z]W \rangle - \langle U, \mathrm{D}\mathcal{M}_X[W]Z \rangle.
\]
By trace rearrangement at $U=V$ and $Z =  G(X)$, we obtain  $\mathcal{M}_X[\allowbreak\Gamma_X(V,  G(X))]=-\Xi_X(V)Q_X$. It follows from $\mathcal{M}_X^{-1}[A] = (I+P_X)A(X^\top X)$ that $\Gamma_X(V,  G(X)) = \allowbreak -(I+P_X)\Xi_X(V)$. Finally, the connection formula $\bar\nabla_V  G = \mathrm{D}G(X)[V] + \Gamma_X(V, G(X))$ yields $\bar\nabla_VG = \mathrm{D}G(X)[V] - (I+P_X)\Xi_X(V)$.

It remains to compute the directional derivative $\mathrm{D}G(X)[V]$. Applying the product rule to $G(X) = 2\myskew(\nabla f(X)X^\top)X$ and substituting the resulting expression alongside $\bar\nabla_V G = \mathrm{D}G(X)[V] + \Gamma_X(V, G(X))$ into $\mathrm{Hess}f(X)[V] = \Pi_X^{\mathrm{T}}(\bar{\nabla}_V G)$ directly yields~\eqref{eq:Riemannian_Hessian}.
\end{proof}

We consider the Riemannian Newton (RN) equation on the layered manifold
\begin{equation}
\label{eq:newton_eq}
\Hess f(X)[T]=-\grad f(X),
\end{equation}
where $\grad f(X)$ and $\Hess f(X)$ are given in~\eqref{eq:Riemannian_gradient} and~\eqref{eq:Riemannian_Hessian}, respectively. The solution to this equation yields the RN update $T_{\mathrm{RN}}(X)$. Together with the Newton--Schulz normal update~\eqref{eq:normal_component}, it may be tempting to consider the iteration
\begin{equation}
\label{eq:Riemannian_newton_landing_field}
\begin{aligned}
\Lambda_{\mathrm{RN}}(X)&=T_{\mathrm{RN}}(X)+N(X),\\
X_{k+1}&=X_k+\Lambda_{\mathrm{RN}}(X_k).
\end{aligned}
\end{equation}
However, as shown by the empirical convergence behavior in~\cref{fig:RN_vs_N}, this iteration fails to achieve the anticipated local quadratic convergence, which motivates us to consider a modified Newton equation.

\subsection{A modified Newton equation}
\label{subsec:symmetric_tangent_system}
The empirical behavior of the naive update $\Lambda_{\mathrm{RN}}$ in~\eqref{eq:Riemannian_newton_landing_field} suggests that the plain Newton equation~\eqref{eq:newton_eq} is not fully adequate for the second-order landing framework. The first-order optimality of~\eqref{eq:problem} demands that $\grad f(X)=0$ and $X\in\St(p,n)$. However, the right-hand side~\eqref{eq:newton_eq} is not able to reflect the infeasibility of $X$ when $X\notin\St(p,n)$ and does not encode the normal component. 
To remedy this, recall~\eqref{eq:Riemannian_gradient} and~\eqref{eq:relative_grad}, namely 
\[
\grad f(X)=T_1(X) = 2\,\myskew(\nabla f(X)X^\top)X,
\]
and consider a movement along the normal space
\[T_1(X+N(X))\approx T_1(X)+\D T_1(X)[N(X)].\]
Note that $\D T_1(X)[N(X)]$ acts as a first-order perturbation of the Riemannian gradient induced by the normal component. Hence, $\D T_1(X)[N(X)]$ can play a role in compensating the RN equation~\eqref{eq:newton_eq} in the second-order landing framework. We start from this perturbation and develop a correction operator $\mathcal A_\N$ that maps onto the tangent space.

Specifically, given a normal vector $V\in\N_Y\mathrm{St}_{X^\top X}(p,n)$, we have 
\begin{equation}
    \label{eq:frechet_G}
    \D T_1(X)[V] = 2\,\myskew\bigl(\nabla^2 f(X)[V]X^\top + \nabla f(X)V^\top\bigr)X + 2\,\Omega(X)V,
\end{equation}
where $\Omega(X)$ is defined by
\begin{equation}
\label{eq:Omega_definition}
    \Omega(X):=\myskew(\nabla f(X)X^\top).
\end{equation}
It follows from~\eqref{eq:tangent_space} that the first term lies in the tangent space while the second term does not. By simply removing the second term, we construct the following correction operator 
\begin{equation}
\label{eq:A_N-def}
\mathcal A_{\N}(X)[V] := 2\,\myskew\!\bigl(\nabla^2 f(X)[V]X^\top + \nabla f(X)V^\top\bigr)X,
\end{equation}
in which the range is confined to the tangent space for all normal vectors $V$. 
For a reason that will become apparent in~\eqref{eq:Omega-bound-new}, $\mathcal{A}_{\N}(X)[N(X)]$ captures the first-order perturbation of the Riemannian gradient by the normal component $N(X)$ within the tangent space.

By incorporating the normal component into the RN equation~\eqref{eq:newton_eq}, we obtain a modified Newton equation in the unknown $T\in\T_X\mathrm{St}_{X^\top X}(p,n)$:
\begin{equation}\label{eq:SOL-sym}
    \Hess f(X)[T] = -\grad f(X)-\mathcal A_{\N}(X)[N(X)].
\end{equation}
It is worth noting that~\eqref{eq:SOL-sym} remains a valid equation on the tangent space of the layered manifold~$\mathrm{St}_{X^\top X}(p,n)$ since both sides belong to the tangent space. The solution to this linear system provides a candidate tangent component $T(X)\in\T_X\mathrm{St}_{X^\top X}(p,n)$ for the second-order landing framework~\eqref{eq:second_order_landing_framework}, which yields a locally quadratically convergent method; see SOL-sym in~\cref{alg:second-order_landing}.

In practice, solving~\eqref{eq:SOL-sym} requires computing the Riemannian Hessian~\eqref{eq:Riemannian_Hessian}, which involves projection and matrix inversions through $(X^\top X)^{-1}$. The next subsection unveils a computationally efficient approximation to~\eqref{eq:SOL-sym}, which is projection-free and inverse-free.

\subsection{An approximation to the modified Newton equation}
\label{subsec:corrected_tan_direction}
To obtain a more scalable implementation, we derive an approximation to the modified Newton equation~\eqref{eq:SOL-sym} by eliminating the projection and matrix inversions involved in the Riemannian Hessian~\eqref{eq:Riemannian_Hessian}.

A careful inspection of~\eqref{eq:Riemannian_Hessian} reveals that for all $V\in\T_X\mathrm{St}_{X^\top X}(p,n)$, the 
term
\begin{equation}
\label{eq:A_T-def}
\mathcal A_{\T}(X)[V]
:=
2\,\myskew\!\bigl(\nabla^2 f(X)[V]X^\top
+ \nabla f(X)V^\top\bigr)X
\end{equation}
belongs to the tangent space~\eqref{eq:tangent_space}, and thus the projection $\Pi_X^{\mathrm{T}}$ leaves $\mathcal{A}_\T(X)$ unchanged. Note that both expressions of $\mathcal A_\T(X)$ and $\mathcal A_\N(X)$ are the first term of $\D T_1(X)$ in~\eqref{eq:frechet_G} but with different domain $\T_Y\mathrm{St}_{X^\top X}(p,n)$ and $\N_Y\mathrm{St}_{X^\top X}(p,n)$, and we make this distinction intentionally.   Consequently, the Riemannian Hessian~\eqref{eq:Riemannian_Hessian} reduces to
\begin{equation*}
\Hess f(X)[V]
= \mathcal A_\T(X)[V] + \Pi_X^{\mathrm{T}}\Big(
2\,\myskew(\nabla f(X)X^\top)V
-(I+P_X)\,\Xi_X(V)
\Big).
\end{equation*}
While $\mathcal{A}_\T(X)$ is straightforward to compute, the remaining component involves the projection $\Pi_X^{\mathrm{T}}$ and matrix inversion through~\eqref{eq:projection_T} and~\eqref{eq:Xi_term}. For computational efficiency, we remove the remaining component and approximate $\Hess f(X)$ by the tangent operator $\mathcal A_\T(X)$, i.e.,
\[
\Hess f(X)[V]\approx \mathcal A_\T(X)[V],
\]
which preserves the second-order information of the objective while avoiding both explicit projections and matrix inversions.

Subsequently, we integrate the Riemannian Hessian approximation $A_\T(X)$ into the modified Newton equation~\eqref{eq:SOL-sym}, and propose an approximate Newton equation in the unknown $T\in\T_X\mathrm{St}_{X^\top X}(p,n)$:
\begin{equation}
\label{eq:computable-tangent-system}
\mathcal{A}_{\T}(X)[T] = -\grad f(X)-\mathcal A_\N(X)[N(X)].
\end{equation}
The operators $\mathcal{A}_{\T}$ and $\mathcal{A}_\N$ are given in~\eqref{eq:A_T-def} and~\eqref{eq:A_N-def}, respectively. The Riemannian gradient ($\grad f$) can be computed from~\eqref{eq:Riemannian_gradient} and a normal component $N(X)$ can be selected from~\cref{sec:design_normal_term}. We have thus obtained a projection-free and inverse-free non-symmetric equation. We employ the solution to this equation as a tangent component $T(X)$ in the second-order landing framework~\eqref{eq:second_order_landing_framework}, yielding another locally quadratically convergent method; see SOL in~\cref{alg:second-order_landing}.

\section{The second-order landing (SOL) method}
\label{sec:SOL_methods}
With the normal and tangent components developed in~\cref{sec:design_normal_term} and~\cref{sec:design_tangent_term}, we concretize the second-order landing framework~\eqref{eq:second_order_landing_framework}, in which the approximate Newton equation~\eqref{eq:computable-tangent-system} or the modified Newton equation~\eqref{eq:SOL-sym} is solved exactly, and then we introduce an inexact variant for practical implementation.  The convergence analysis establishes a local quadratic convergence and a superlinear convergence for the inexact variant.

\subsection{SOL and SOL-sym methods}
\label{subsec:sol_algorithm}
Now we have all the ingredients to fulfill the second-order landing framework~\eqref{eq:second_order_landing_framework}. Given an iterate $X_k$, we consider the second-order landing update
\begin{equation}\label{eq:second_order_landing_field}
    \Lambda(X_k) := T(X_k) + N(X_k),
\end{equation}
where $T(X_k)$ is obtained by solving~\eqref{eq:computable-tangent-system} or~\eqref{eq:SOL-sym}, and 
$N(X_k) = N_1(X_k)$ is the order-$1$ Newton--Schulz update given by~\eqref{eq:normal_component}. 

To ensure that the second-order landing iteration 
\[X_{k+1} = X_k + \eta_k \Lambda(X_k)\]
remains within a local safe region $\St(p,n)^\varepsilon$, we adopt a safeguard step-size rule.
A general safeguard for landing-type methods was given in~\cite[Lem.~3]{ablinInfeasibleDeterministicStochastic2024}; the following lemma tailors this safeguard for the SOL and SOL-sym methods.

\begin{lemma}\label{prop:safe_step_size_NL}
Let $X_k\in \St(p,n)^\varepsilon$ and consider the update
$X_{k+1}=X_k+\eta_k\Lambda(X_k)$, where $\Lambda(X_k)$ is defined by
\eqref{eq:second_order_landing_field} and $\eta_k>0$ is the step size. Let
$g:=\normF{\Lambda(X_k)}$ and $d:=\normF{X_k^\top X^{}_k-I_p }$. If
$\eta_k \leq \eta_\mathrm{safe}(X_k)$ with
\begin{equation}
    \label{eq:safe_stepsize}
    \eta_\mathrm{safe}(X_k):=\min\left\{\frac{\frac12 d(1-d)+\sqrt{\frac14d^2(1-d)^2+g^2(\varepsilon-d)}}{g^2},1\right\},
\end{equation}
then $X_{k+1}\in \St(p,n)^\varepsilon$.
\end{lemma}

\begin{proof}
This is~\cite[Lem.~3]{ablinInfeasibleDeterministicStochastic2024} specialized to $\lambda=1/2$.
\end{proof}

\begin{algorithm}[htbp]
\caption{Second-order landing method (SOL and SOL-sym) }
\label{alg:second-order_landing}
\begin{algorithmic}[1]
  \REQUIRE Initial point $X_0 \in \St(p,n)^\varepsilon$,
           maximum number of iterations \texttt{mxit},
           stopping tolerance $\texttt{tol}>0$. Set $k = 0$.
  \WHILE{$\normF{\grad f(X_k)} + \normF{ N(X_k)} > \texttt{tol}$ and $k < \texttt{mxit}$}
      \STATE Compute the order-$1$ Newton--Schulz update
           $ N(X_k) = N_1(X_k)$, see~\eqref{eq:normal_component}.
    \STATE Compute $T(X_k)$ by solving~\eqref{eq:computable-tangent-system} (or~\eqref{eq:SOL-sym} for SOL-sym).    \label{alg:line_tangent}
    \STATE Compute $\widehat{X}_{k+1} = X_k+\Lambda(X_k)$ with $\Lambda(X_k)=T(X_k)+N(X_k)$.
\IF{$\normF{\widehat{X}_{k+1}^\top\widehat{X}_{k+1}-I_p}\leq\varepsilon$}
    \STATE Accept the unit step: $X_{k+1} = \widehat{X}_{k+1}$.
\ELSE
    \STATE Choose $\eta_k\le \eta_\mathrm{safe}(X_k)$ from~\eqref{eq:safe_stepsize} and take a safe step: $X_{k+1} = X_k+\eta_k\Lambda(X_k)$.
\ENDIF
    \STATE Set $k=k+1$.
  \ENDWHILE
  \ENSURE $X_k$.
\end{algorithmic}
\end{algorithm}

The resulting second-order landing method is summarized in~\cref{alg:second-order_landing}. Unless otherwise stated, $T(X_k)$ is computed exactly. An inexact variant is proposed in~\cref{subsec:inexact_sol_algorithm}, which enjoys superlinear convergence. Specifically, starting from an initial point in the safe region, each iteration first constructs the normal component $N(X_k)$ from the order-$1$ Newton--Schulz iteration and computes
the tangent component $T(X_k)$ by solving either the approximate Newton equation~\eqref{eq:computable-tangent-system} for SOL, or the modified Newton equation~\eqref{eq:SOL-sym} for SOL-sym. The suffix ``sym'' emphasizes that SOL-sym exploits the symmetry of the complete Riemannian Hessian operator $\Hess f(X_k)$ with respect to the metric $g$.
After forming the second-order landing update $\Lambda(X_k)=T(X_k)+N(X_k)$, the algorithm tries a unit step $X_k+\Lambda(X_k)$. If it falls outside the safe region $\St(p,n)^\varepsilon$, a safeguarded step size is computed via~\cref{prop:safe_step_size_NL}. It is worth noting that the safeguard mechanism becomes inactive near the solution. The algorithm terminates when the first-order optimality is achieved within the prescribed tolerance.

\subsection{Local quadratic convergence of the SOL method}
\label{section:convergence_analysis}

We now analyze the local convergence of SOL (\cref{alg:second-order_landing}). As shown in~\cref{section:Pre}, the zeros of the first-order landing field $\Lambda_1(X)=T_1(X) + \lambda \nabla \mathcal{N}(X)$ are the first-order stationary points of~\eqref{eq:problem}. This motivates a Newton-type local convergence analysis through the lens of the equation 
\begin{equation}
    \label{eq:analysis_equation}
    \Lambda_1(X)=0.
\end{equation}
Specifically, we show that $\Lambda(X)$ in SOL can be viewed as an approximate Newton step for solving~\eqref{eq:analysis_equation} in a neighborhood of a nondegenerate minimizer.

We begin by stating the smoothness and nondegeneracy assumptions used in the local analysis.

\begin{assumption}\label{asm:smoothness}
We assume that:
\\
(i) $f$ is twice continuously differentiable and $L_0$-Lipschitz in $\St(p,n)^\varepsilon$.
\\
(ii) The Euclidean gradient $\nabla f$ is $L_1$-Lipschitz continuous in $\St(p,n)^\varepsilon$.
\\
(iii) The Euclidean Hessian $\nabla^2 f$ is locally $L_2$-Lipschitz continuous at $X_\star$, i.e., 
  there exist $L_2 > 0$ and $\delta_0 > 0$ such that
  \[
    \|\nabla^2 f(X) - \nabla^2 f(Y)\|
    \le L_2 \normF{X - Y},
    \qquad \text{for all } X,Y \in B_{\delta_0}(X_\star),
  \]
  where $B_r(X_\star)$ denotes the open ball of radius $r$ centered at $X_\star$.
\end{assumption}

For the local analysis, we assume that the target local minimizer satisfies the second-order sufficient optimality condition.

\begin{assumption}[nondegenerate local minimizer]\label{asm:regularity}
Let $X_\star\in \St(p,n)$ be a local minimizer of~\eqref{eq:problem}.
Assume that $X_\star$ is a nondegenerate local minimizer, i.e., the Riemannian Hessian defined in~\eqref{eq:Riemannian_Hessian} at $X_\star$ is
positive definite on the tangent space.
\end{assumption}

For subsequent proofs, it is useful to establish basic bounds on iterates within the safe region.

\begin{lemma}
\label{lem:bounds}
For all $X\in\mathrm{St}(p,n)^{\varepsilon}$, the following bounds hold:
\[
\|X\|_{2}\leq\sqrt{1+\varepsilon},\quad\|(X^{\top}X)^{-1}\|_{2}\leq\frac{1}{1-\varepsilon}.
\]
\end{lemma}

\begin{proof}
By~\cite[Lem.~1]{ablinInfeasibleDeterministicStochastic2024}, any $X\in \St(p,n)^\varepsilon$ satisfies
$\sqrt{1-\varepsilon}\ \le\ \sigma_{\min}(X)\ \le\ \sigma_{1}(X)\ \le\ \sqrt{1+\varepsilon}.$
Therefore, $\|X\|_2=\sigma_{1}(X)\le \sqrt{1+\varepsilon}$. Moreover, since the singular values of $X^\top X$ are $\{\sigma_i(X)^2\}$, we have
$\sigma_{\min}(X^\top X)=\sigma_{\min}(X)^2\ge 1-\varepsilon.$
Hence $X^\top X$ is invertible and
$\|(X^\top X)^{-1}\|_2=\frac{1}{\sigma_{\min}(X^\top X)}\le \frac{1}{1-\varepsilon}.$
\end{proof}

In order to analyze~\eqref{eq:analysis_equation}, we explicitly derive the Jacobian $\mathcal{J}_{\Lambda_1}(X)$ of the landing field $\Lambda_1(X)$. Since the SOL method splits the update into tangent and normal components, we analyze the Jacobian structure with respect to the direct sum decomposition \[\mathbb{R}^{n\times p}=\T_{X}\StXTX(p,n)\oplus \N_{X}\StXTX(p,n).\]

\begin{proposition}[Jacobian of the landing field $\Lambda_1(X)$]
\label{prop:block-jac}
Let $\Pi_X^{\mathrm{T}}$ be the $g$-orthogonal projection from $\R^{n\times p}$ onto \(\T_X\StXTX(p,n)\) given in~\eqref{eq:projection_T}, and define the normal projection onto $\N_X\StXTX(p,n)$ as $\Pi_X^{\mathrm{N}}:=\Id-\Pi_X^{\mathrm{T}}$, where $\Id$ is the identity map.
Then the Jacobian $\mathcal{J}_{\Lambda_1}(X)$ of the landing field
$\Lambda_1(X)=2\,\mathrm{skew}(\nabla f(X)X^\top)X+\lambda\,\nabla\mathcal N(X)$
admits the following block operator representation,
\[
\mathcal{J}_{\Lambda_1}(X)=
\begin{bmatrix}
 J_{\T\T}(X) & J_{\N\T}(X)\\
 J_{\T\N}(X) & J_{\N\N}(X)
\end{bmatrix}
\]
with the block actions, for all $V=V_\T+V_\N\in\R^{n \times p}$ where
$V_\T\in \T_X\StXTX(p,n)$ and $V_\N\in \N_X\StXTX(p,n)$:
\begin{align*}
J_{\T\T}(X)[V_\T]={}&
2\mathrm{skew}\big(\nabla^2 f(X)[V_\T]X^\top + \nabla f(X)V_\T^\top\big)X
\\
\notag
&+\Pi_X^{\mathrm{T}}\big(2\,\mathrm{skew}(\nabla f(X) X^\top)V_\T + \lambda\,V_\T\,(X^\top X-I_p)\big),
\\
J_{\T\N}(X)[V_\T]={}&
\Pi_X^{\mathrm{N}}\!\Big(2\mathrm{skew}(\nabla f(X) X^\top)\,V_\T
\;+\;\lambda\,V_\T\,(X^\top X-I_p)\Big),
\\
J_{\N\T}(X)[V_\N]
={}&
2\mathrm{skew}\!\big(\nabla^2 f(X)[V_\N]X^\top + \nabla f(X)V_\N^\top\big)X 
\\\notag &+\Pi_X^{\mathrm{T}}\!\big(
2\mathrm{skew}(\nabla f(X)X^\top)V_\N+\lambda V_\N(X^\top X-I_p)\big)
\\
\notag
&+\Pi_X^{\mathrm{T}}\!\Big( 2\lambda X\,\mathrm{sym}(X^\top V_\N)
\Big),
\\
J_{\N\N}(X)[V_\N]
={}&
\Pi_X^{\mathrm{N}}\!\Big(
2\,\mathrm{skew}(\nabla f(X)X^\top)V_\N +\lambda V_\N(X^\top X-I_p)\Big)
\\\notag
& + \Pi_X^{\mathrm{N}}\!\Big(2\lambda X\,\mathrm{sym}(X^\top V_\N)
\Big).
\end{align*}

\end{proposition}

\begin{proof}
    Differentiating $\Lambda_1$ at $X$ along $V$ yields the Jacobian,
    \begin{equation}\label{eq:fullJac}
    \begin{aligned}
        \mathcal{J}_{\Lambda_1}(X)[V] =\;& 
        \underbrace{2\,\mathrm{skew}\big(\nabla^2 f(X)[V]X^\top + \nabla f(X)V^\top\big)X}_{\in \T_X\StXTX(p,n)} 
        \;+\; 2\,\mathrm{skew}(\nabla f(X)X^\top)V \\
        & +\; \lambda V(X^\top X - I_p) \;+\; 2\lambda X\,\mathrm{sym}(X^\top V).
    \end{aligned}
    \end{equation}
    We derive the block components by restricting $V$ to the tangent space $\T_X\StXTX(p,n)$ 
    or the normal space $\N_X\StXTX(p,n)$ and applying the $g$-orthogonal projections $\Pi_X^{\mathrm{T}}$ and $\Pi_X^{\mathrm{N}}$.
    
    For a tangent vector $V_\T \in \T_X\StXTX(p,n)$, the property $\mathrm{sym}(X^\top V_\T) = 0$ eliminates the last term in~\eqref{eq:fullJac}. The first term lies naturally in $\T_X\StXTX(p,n)$ (as it is of the form $KX$ with $K$ skew-symmetric); it is preserved by $\Pi_X^{\mathrm{T}}$ and canceled by $\Pi_X^{\mathrm{N}}$. Applying the projections to the remaining terms directly yields $J_{\T\T}$ and $J_{\T\N}$.
    
    For a normal vector $V_\N \in \N_X\StXTX(p,n)$, the first term in~\eqref{eq:fullJac}
    still lies in $\T_X\StXTX(p,n)$.
    The remaining terms are, in general, neither purely tangent nor purely normal.
    Therefore, they must be decomposed by applying the projections $\Pi_X^{\mathrm{T}}$ and $\Pi_X^{\mathrm{N}}$.
    This yields the expressions for $J_{\N\T}$ and $J_{\N\N}$.
\end{proof}

The Jacobian can be simplified at a minimizer as follows.

\begin{corollary}[Jacobian at stationary points]
\label{cor:block-upper-Jac}
Let $X_\star$ be a stationary point for problem~\eqref{eq:problem}. Then the Jacobian $\mathcal{J}_{\Lambda_1}(X_\star)$ admits a block upper-triangular structure with respect to the direct sum decomposition $\mathbb R^{n\times p}=\T_{X_\star}\St(p,n)\oplus \N_{X_\star}\St(p,n)$:
\[
\mathcal{J}_{\Lambda_1}(X_\star)
=
\begin{bmatrix}
J_{\T\T}(X_\star) & J_{\N\T}(X_\star) \\
\mathbf{0} & J_{\N\N}(X_\star)
\end{bmatrix}.
\]
Specifically, for all $V_\T\in \T_{X_\star}\St(p,n)$ and $V_\N\in \N_{X_\star}\St(p,n)$, the nonzero blocks simplify to
\begin{equation*}
    \begin{aligned}
    J_{\T\T}(X_\star)[V_\T] &= 2\,\mathrm{skew}\big(\nabla^2 f(X_\star)[V_\T]\,X_\star^\top + \nabla f(X_\star)\,V_\T^\top\big)\,X_\star = \Hess f(X_\star)[V_\T],
    \\
    J_{\N\T}(X_\star)[V_\N] &= 2\,\mathrm{skew}\big(\nabla^2 f(X_\star)[V_\N]\,X_\star^\top + \nabla f(X_\star)\,V_\N^\top\big)\,X_\star,
    \\
    J_{\N\N}(X_\star)[V_\N] &= 2\lambda\,V_\N.
    \end{aligned}
\end{equation*}
\end{corollary}

\begin{proof}
Since $X_\star$ is a stationary point, we have
\[
X_\star^\top X_\star=I_p,
\qquad
\myskew(\nabla f(X_\star)X_\star^\top)=0.
\]
Substituting them into $\mathcal{J}_{\Lambda_1}(X_\star)$ of \cref{prop:block-jac} gives
$J_{\T\N}(X_\star)=0$ and yields $J_{\T\T}(X_\star)$ and $J_{\N\T}(X_\star)$.

It remains to simplify $J_{\N\N}(X_\star)$. For $V_\N\in\N_{X_\star}\St(p,n)$, write
$V_\N=X_\star S$ with $S\in\mathcal{S}_{\sym}(p)$. Then
\[
2\lambda X_\star\sym(V_\N^\top X_\star)
=
2\lambda X_\star\sym(SX_\star^\top X_\star)
=
2\lambda X_\star S
=
2\lambda V_\N.
\]
Since $V_\N$ already lies in the normal space, applying $\Pi_\N^{X_\star}$ leaves it unchanged, and we obtain $J_{\N\N}(X_\star)$.
\end{proof}

We also establish the Lipschitz continuity of the Jacobian, which is essential for quadratic convergence analysis.

\begin{lemma}[smoothness of the Jacobian $\mathcal{J}_{\Lambda_1}(X)$]
\label{lem:smoothness-JLambda}
Suppose \cref{asm:smoothness} holds. Then the Jacobian of $\Lambda_1$ is Lipschitz continuous on
$B_{\delta_0}(X_\star)\cap \St(p,n)^\varepsilon$ with respect to the operator norm. Specifically, for all
$X,Y\in B_{\delta_0}(X_\star)\cap \St(p,n)^\varepsilon$,
\begin{equation*}\label{eq:Lipschitz-JLambda-const}
\bigl\|\mathcal{J}_{\Lambda_1}(X)-\mathcal{J}_{\Lambda_1}(Y)\bigr\|
\;\le\; L_2'\,\normF{X-Y},
\end{equation*}
where $L_2' := 2(1+\varepsilon)L_2
+ 8\sqrt{1+\varepsilon}\,L_1
+ 4L_0
+ 6\lambda\sqrt{1+\varepsilon}$.
\end{lemma}

\begin{proof}
    Let $X,Y \in B_{\delta_0}(X_\star)\cap \St(p,n)^\varepsilon$ and set $\Delta := X-Y$.
    We rely on the bounds $\norm{X}_2, \norm{Y}_2 \le \sqrt{1+\varepsilon}$, $\normF{\nabla f(X)} \le L_0$, and $\|\nabla^2 f(X)\| \le L_1$ (from \cref{asm:smoothness} and \cref{lem:bounds}). 
    We also utilize the Lipschitz property of the quadratic form on the safe region:
    \begin{equation}
    \label{eq:Lip_XTX}
    \begin{aligned}
    \normF{X^\top X-Y^\top Y}
    & = \normF{X^\top(X-Y)+(X-Y)^\top Y}
    \\
    &\leq(\|X\|_2+\|Y\|_2)\normF{X-Y}\leq2\sqrt{1+\varepsilon}\normF{\Delta},
    \end{aligned}
    \end{equation}
    where the first inequality comes from the triangle inequality and that $\normF{AB}\leq\norm{A}_2\normF{B}$ for all matrices $A$ and $B$.
    
    Recall the Jacobian from~\eqref{eq:fullJac}:
    \begin{equation*}
    \begin{aligned}
        \mathcal{J}_{\Lambda_1}(X)[V] =\;& 
        \underbrace{2\,\mathrm{skew}\!\big(\nabla f(X)V^\top\big)X + 2\,\mathrm{skew}\!\big(\nabla f(X)X^\top\big)V}_{=:\,T_G(X)[V]} \\
        & +\;\underbrace{2\,\mathrm{skew}\!\big(\nabla^2 f(X)[V]\,X^\top\big)\,X}_{=:\,T_H(X)[V]} 
       \;+\; \underbrace{\lambda V(X^\top X-I_p)+2\lambda X\,\mathrm{sym}(X^\top V)}_{=:\,T_P(X)[V]}.
    \end{aligned}
    \end{equation*}
    Considering a unit-norm $V$ ($\normF{V}=1$),
    we proceed to analyze the Lipschitz continuity of each component separately. For the Hessian-related term $T_H$, the difference is bounded as follows,
    \begin{align*}
        \|T_H(X)&[V] - T_H(Y)[V]\|_{\frob}\\
        &\le 2 \normF{\mathrm{skew}\big(\nabla^2 f(X)[V] X^\top\big) (X-Y)} \\
        &\quad + 2 \normF{\mathrm{skew}\big( (\nabla^2 f(X) - \nabla^2 f(Y))[V] X^\top \big) Y} \\
        &\quad + 2 \normF{\mathrm{skew}\big( \nabla^2 f(Y)[V] (X - Y)^\top \big) Y} \\
        &\le 2 \left( L_1 \sqrt{1+\varepsilon} \normF{\Delta} + L_2 \normF{\Delta} (1+\varepsilon) + L_1 \normF{\Delta} \sqrt{1+\varepsilon} \right) \\
        &= \Big( 2(1+\varepsilon)L_2 + 4\sqrt{1+\varepsilon}L_1 \Big) \normF{\Delta}.
    \end{align*}
    In the first inequality, we apply the triangle inequality. In the second inequality, we invoke the property $\normF{\mathrm{skew}(A)} \le \normF{A}$ and $\normF{AB}\le\normF{A}\norm{B}_2$ for all matrices $A$ and $B$, followed by substituting the Lipschitz constant $L_2$ for the Hessian and the bounds $L_1$ and $\sqrt{1+\varepsilon}$.
    
    Next, for the gradient-related term $T_G$, we have
    \begin{align*}
        \|T_G(X)&[V] - T_G(Y)[V]\|_{\frob}\\
        &\le 2 \normF{\mathrm{skew}\big( (\nabla f(X) - \nabla f(Y)) V^\top \big) X} + 2 \normF{\mathrm{skew}\big( \nabla f(Y) V^\top \big) (X - Y)} \\
        &\quad + 2 \normF{\mathrm{skew}\big( (\nabla f(X) - \nabla f(Y)) X^\top \big) V} + 2 \normF{\mathrm{skew}\big( \nabla f(Y) (X - Y)^\top \big) V} \\
        &\le 2 \Big( L_1 \normF{\Delta} \sqrt{1+\varepsilon} + L_0 \normF{\Delta} \Big) + 2 \Big( L_1 \normF{\Delta} \sqrt{1+\varepsilon} + L_0 \normF{\Delta} \Big) \\
        &= \Big( 4\sqrt{1+\varepsilon}L_1 + 4L_0 \Big) \normF{\Delta}.
    \end{align*}
    The first step applies the same expansion fact to both summands of $T_G$. The second step utilizes the Lipschitz constant $L_1$ of the Euclidean gradient and the bound $L_0$ of the Euclidean gradient on the safe region, along with $\normF{\mathrm{skew}(A)} \le \normF{A}$.
    
     Finally, for the penalty-related term $T_P$, the variation is bounded by
    \begin{align*}
        \|T_P(X)&[V] - T_P(Y)[V] \|_{\frob}\\
        &\le \lambda \normF{V (X^\top X - Y^\top Y)} + 2\lambda \normF{X \mathrm{sym}(X^\top V) - Y \mathrm{sym}(Y^\top V)} \\
        &\le   2 \lambda\sqrt{1+\varepsilon} \normF{\Delta}  + 2\lambda \left( \normF{(X-Y)\mathrm{sym}(X^\top V)} + \normF{Y \mathrm{sym}((X-Y)^\top V)} \right) \\
        &\le 2\lambda \sqrt{1+\varepsilon} \normF{\Delta} + 2\lambda \Big( \normF{\Delta} \sqrt{1+\varepsilon} + \sqrt{1+\varepsilon} \normF{\Delta} \Big) \\
        &= 6\lambda\sqrt{1+\varepsilon} \normF{\Delta}.
    \end{align*}
    The first step follows from the triangle inequality. The second step applies~\eqref{eq:Lip_XTX} and expands the bilinear term. The third step bounds the variations using $\norm{X}_2 \le \sqrt{1+\varepsilon}$ and $\normF{\mathrm{sym}(A)} \le \normF{A}$.
    
    Summing the constants from the three components yields the Lipschitz constant $L_2'$.
\end{proof}

In the context of solving $\Lambda_1(X)=0$, the Newton step $\Delta X$ at $X$ is defined as the solution to the linear system \[\mathcal{J}_{\Lambda_1}(X)[\Delta X] = -\Lambda_1(X).\]
In contrast, SOL does not work directly with the exact Jacobian $\mathcal{J}_{\Lambda_1}(X)$. Instead, it employs an approximate Jacobian, denoted by $\mathcal{A}(X)$, to define the update. We formalize this operator below.

\begin{definition}[second-order landing operator]
\label{def:Approximate-Jacobian}
    Define the second-order landing operator 
$\mathcal{A}(X):\mathbb{R}^{n\times p}\to\mathbb{R}^{n\times p}$ 
with respect to the direct sum decomposition $\mathbb{R}^{n\times p}=\T_{X}\StXTX(p,n)\oplus\N_{X}\StXTX(p,n)$ as the block operator
\begin{equation}
    \label{eq:second_order_landing_operator}
\mathcal{A}(X)=
    \begin{bmatrix}
 \mathcal{A}_{\T}(X) & \mathcal{A}_{\N}(X)\\[1mm]
\mathbf{0} & 2\lambda\mathrm{Id}
\end{bmatrix},
\end{equation}
where $\Id$ is the identity map and the top-left and top-right blocks are defined by~\eqref{eq:A_T-def} and~\eqref{eq:A_N-def}, respectively.
\end{definition}

Through the second-order landing operator, the update $\Lambda(X)$ of SOL constructed in~\eqref{eq:second_order_landing_field} satisfies the approximate Newton equation $\mathcal{A}(X)[\Lambda(X)] = -\Lambda_1(X)$. Specifically, by exploiting the block upper-triangular structure of $\mathcal{A}(X)$, the system decouples into systems for the tangent and normal components:
\begin{equation*}
\mathcal{A}(X)[\Lambda(X)] = -\Lambda_1(X)
\quad\iff\quad
\left\{ % 
\begin{aligned}
\mathcal{A}_{\T}(X)[T(X)] + \mathcal{A}_{\N}(X)[N(X)] &= -T_1(X), \\
2\lambda N(X) &= -\lambda \nabla \mathcal{N}(X).
\end{aligned}
\right. 
\end{equation*}
The first equation rearranges to $\mathcal{A}_{\T}(X)[T(X)] = -T_1(X) - \mathcal{A}_{N}(X)[N(X)]$, which matches the tangent linear system~\eqref{eq:computable-tangent-system} in SOL, while
the second equation corresponds to the order-$1$ Newton--Schulz normal update $N(X) = -\frac{1}{2}\nabla \mathcal{N}(X) = N_1(X)$ defined in~\eqref{eq:normal_component}. By~\cref{cor:block-upper-Jac}, we observe that \[\mathcal{A}(X_\star) = \mathcal{J}_{\Lambda_1}(X_\star),\] 
which implies that the approximate Jacobian is accurate to the first order at the solution. The following result confirms that the discrepancy between the two operators vanishes linearly with the distance to the solution.

\begin{proposition}\label{prop:Ldelta-explicit}
Suppose \cref{asm:smoothness,asm:regularity} hold. For all $X\in B_{\delta_0}(X_\star)\cap \St(p,n)^\varepsilon$,
\[
\bigl\|\mathcal{A}(X)-\mathcal{J}_{\Lambda_1}(X)\bigr\|
\;\le\; L_{\Delta}\,\normF{X-X_\star},
\]
where $L_{\Delta}:=2\bigl(\sqrt{1+\varepsilon}\,L_1+L_0\bigr)
+2\lambda\sqrt{1+\varepsilon}\,\frac{3+\varepsilon}{1-\varepsilon}$.
\end{proposition}

\begin{proof}
Let $\Delta:=X-X_\star$, $E:=X^\top X-I_p$, and recall from~\eqref{eq:Omega_definition} that $\Omega(X)=\myskew(\nabla f(X)X^\top)$.
Since $\Omega(X_\star)=0$, we have
\begin{align}
\normF{\Omega(X)}
&= \normF{\myskew(\nabla f(X)X^\top-\nabla f(X_\star)X_\star^\top)} \leq \normF{\nabla f(X)X^\top-\nabla f(X_\star)X_\star^\top}  \notag\\
&\le \normF{\nabla f(X)-\nabla f(X_\star)}\,\|X\|_2
      +\normF{\nabla f(X_\star)}\,\|\Delta\|_2 \notag\\
&\le (\sqrt{1+\varepsilon}\,L_1+L_0)\,\normF{\Delta},
\label{eq:Omega-bound-new}
\end{align}
where we use the Lipschitz continuity of $\nabla f$, $\|\Delta\|_2\le\normF{\Delta}$,
and $\|X\|_2\le \sqrt{1+\varepsilon}$.
Moreover,
\begin{equation}
\label{eq:E-bound-new}
    \normF{E}
=\normF{X^\top X-X_\star^\top X_\star}
\le (\|X\|_2+\|X_\star\|_2)\,\normF{\Delta}
\le 2\sqrt{1+\varepsilon}\,\normF{\Delta}.
\end{equation}

Let $\mathcal D(X):=\mathcal A(X)-\mathcal J_{\Lambda_1}(X)$ and set
$Q_X:=(X^\top X)^{-1}$. For all $V\in\mathbb R^{n\times p}$, write
$V=V_{\T}+V_{\N}$ with
$V_{\N}:=(\Id-\Pi_X^{\T})V=XQ_XS$, where
$S:=\sym(X^\top V)$, $\Id$ is the identity map, and $\Pi_X^{\T}$ is given in~\eqref{eq:projection_T}. Using the definition of $\mathcal A(X)$ in~\eqref{eq:second_order_landing_operator}
and the expression of $\mathcal J_{\Lambda_1}(X)$ in~\eqref{eq:fullJac}, we obtain
\[
\begin{aligned}
\mathcal A(X)[V]
&=
2\myskew\bigl(\nabla^2 f(X)[V]X^\top+\nabla f(X)V^\top\bigr)X
+2\lambda V_\N,\\
\mathcal J_{\Lambda_1}(X)[V]
&=
2\myskew\bigl(\nabla^2 f(X)[V]X^\top+\nabla f(X)V^\top\bigr)X
+2\Omega(X)V+\lambda V E+2\lambda XS .
\end{aligned}
\]
Hence,
\[
\mathcal D(X)[V]
=
-2\Omega(X)V-\lambda V E+2\lambda X(Q_X-I_p)S .
\]
Since $X^\top X=I_p+E$, we have $Q_X-I_p=-Q_XE$. Therefore,
\[
\begin{aligned}
\normF{\mathcal D(X)[V]}
&\leq
2\normF{\Omega(X)}\normF{V}
+\lambda\normF{E}\normF{V}
+2\lambda\|X\|_2\|Q_X\|_2\normF{E}\normF{S} \\
&\leq
\left(
2\normF{\Omega(X)}
+\lambda\left(1+2\frac{1+\varepsilon}{1-\varepsilon}\right)\normF{E}
\right)\normF{V},
\end{aligned}
\]
where we used $\normF S\le \|X\|_2\normF{V}$,
$\|X\|_2\le\sqrt{1+\varepsilon}$, and
$\|Q_X\|_2\leq (1-\varepsilon)^{-1}$ from~\cref{lem:bounds}. It follows that
\[
\|\mathcal D(X)\|
\leq
2\normF{\Omega(X)}
+\lambda\frac{3+\varepsilon}{1-\varepsilon}\normF{E} .
\]
Finally, substituting~\eqref{eq:Omega-bound-new}--\eqref{eq:E-bound-new} into the above estimate yields the claim.
\end{proof}

We next state the main local convergence result.

\begin{theorem}[quadratic convergence]\label{thm:quadratically}
Suppose \cref{asm:smoothness,asm:regularity} hold. Let $(X_k)$ be generated by the unit-step SOL method in~\cref{alg:second-order_landing}, given by
\[
X_{k+1}=X_k+\Lambda(X_k),
\qquad
\Lambda(X_k)=T(X_k)+N(X_k),
\]
where the normal component $N(X_k)$ is given by~\eqref{eq:normal_component}, and the tangent component
$T(X_k)\in \T_{X_k}\St_{X_k^\top X^{}_k}(p,n)$
is the solution to~\eqref{eq:computable-tangent-system}. Then there exists $\bar\delta>0$ such that, for all
$X_0\in B_{\bar\delta}(X_\star)\cap \St(p,n)^\varepsilon$,
the iterates remain in $\St(p,n)^\varepsilon$ and converge to $X_\star$ quadratically. 
\end{theorem}

\begin{proof}
By \cref{asm:regularity} and \cref{cor:block-upper-Jac}, the Jacobian $\mathcal{J}_{\Lambda_1}(X_\star)$ is nonsingular. Let
\[
\kappa_\star := \|\mathcal{J}_{\Lambda_1}(X_\star)^{-1}\|.
\]
Since $\mathcal{A}(X_\star)=\mathcal{J}_{\Lambda_1}(X_\star)$ and $\mathcal{A}$ is continuous, there exists
$\delta_A>0$ such that
$\|\mathcal{A}(X)-\mathcal{A}(X_\star)\| \le (2\kappa_\star)^{-1}$, for all $X\in B_{\delta_A}(X_\star)$.
Hence $\mathcal{A}(X)$ is invertible on $B_{\delta_A}(X_\star)$ and
\begin{equation}
\label{eq:Ainv_bd}
\|\mathcal{A}(X)^{-1}\|\le 2\kappa_\star ,\qquad \text{for~all}\,X\in B_{\delta_A}(X_\star).
\end{equation}

Choose $\delta_F>0$ such that $B_{\delta_F}(X_\star)\subset \St(p,n)^\varepsilon$. Let $\delta > 0$ be the radius associated with \cref{asm:smoothness}, then define
$C:=2\kappa_\star\Big(L_\Delta+\tfrac12 L_2'\Big),
\bar\delta := \min\Big\{\delta_0,\delta_A,\delta_F,\; \frac{1}{C}\Big\}$,
and consider any $X_k\in B_{\bar\delta}(X_\star)$. Let \[e_k:=X_k-X_\star.\]
By the approximate Newton equation $\mathcal{A}(X_k)[\Lambda(X_k)]=-\Lambda_1(X_k)$, we have
\[
e_{k+1}=e_k+\Lambda(X_k)
= e_k-\mathcal{A}(X_k)^{-1}\Lambda_1(X_k)
= \mathcal{A}(X_k)^{-1}\big(\mathcal{A}(X_k)[e_k]-\Lambda_1(X_k)\big).
\]
It follows from the fundamental theorem of calculus that
$\Lambda_1(X_k)
= \int_0^1 \mathcal{J}_{\Lambda_1}(X_\star+t e_k)[e_k]\,\mathrm{d}t$,
hence
\[
\mathcal{A}(X_k)[e_k]-\Lambda_1(X_k)
=\big(\mathcal{A}(X_k)-\mathcal{J}_{\Lambda_1}(X_k)\big)[e_k]
+\int_0^1\big(\mathcal{J}_{\Lambda_1}(X_k)-\mathcal{J}_{\Lambda_1}(X_\star+t e_k)\big)[e_k]\,\mathrm{d}t.
\]
Taking norms and applying \cref{prop:Ldelta-explicit} and \cref{lem:smoothness-JLambda}, we obtain
\begin{align*}
    \normF{\mathcal{A}(X_k)[e_k] - \Lambda_1(X_k)} 
    &\le \norm{\mathcal{A}(X_k) - \mathcal{J}_{\Lambda_1}(X_k)} \,\normF{e_k} + \int_0^1 L_2' (1-t)\normF{e_k}^2 \, \mathrm{d}t \\
    &\le L_{\Delta} \normF{e_k}^2 + \frac{1}{2}L_2' \normF{e_k}^2 = \left( L_{\Delta} + \frac{1}{2}L_2' \right) \normF{e_k}^2.
\end{align*}
Combining this with the inverse bound~\eqref{eq:Ainv_bd}, we obtain the quadratic convergence
\begin{equation}\label{eq:qquad}
    \normF{e_{k+1}} \le 2\kappa_\star \left( L_{\Delta} + \frac{1}{2}L_2' \right) \normF{e_k}^2 = C \normF{e_k}^2.
\end{equation}
Since $\normF{e_k} \le \bar{\delta} \le 1/C$, we have $C\normF{e_k} \le 1$, which implies $\normF{e_{k+1}} \le \normF{e_k} \le \bar{\delta}$. Thus, $X_{k+1}$ remains in $B_{\bar{\delta}}(X_\star)$, allowing the induction to proceed. The inequality~\eqref{eq:qquad} establishes the quadratic convergence.
\end{proof}

\begin{remark}
\label[remark]{remark:higher_order_equivalence}
The equation~\eqref{eq:computable-tangent-system} is developed as a tractable approximation of the modified Newton equation~\eqref{eq:SOL-sym}. Near a nondegenerate minimizer, the discrepancy between the two equations is of higher order. Consequently, the tangent components generated by the two systems, as well as the resulting second-order landing updates, differ only by higher-order terms in the local regime.
For this reason, the proof of local quadratic convergence of SOL via the approximation equation~\eqref{eq:computable-tangent-system} also applies to SOL-sym via the modified Newton equation~\eqref{eq:SOL-sym}, up to minor adjustments of the constants.
\end{remark}

\subsection{An inexact variant of SOL with local superlinear convergence}
\label{subsec:inexact_sol_algorithm}

Although SOL (or SOL-sym) achieves fast local convergence, solving the approximate Newton equation~\eqref{eq:computable-tangent-system} (or the modified Newton equation~\eqref{eq:SOL-sym}) exactly at every iteration is computationally prohibitive in large-scale settings. To alleviate this burden, we develop an inexact variant.

Within the inexact framework, the normal component $N(X_k)$ is evaluated explicitly and remains exact. The inexactness is therefore confined to the tangent component. In practice, the linear system is solved approximately via an appropriate Krylov subspace solver, such as the biconjugate gradient stabilized method (BiCGSTAB)~\cite{VorstBicgstab1992} or LGMRES~\cite{bakerlgmres2005} for the non-symmetric equation~\eqref{eq:computable-tangent-system} in SOL, and CG or MINRES~\cite{paigeSolutionSparseIndefinite1975} under the metric $g$ defined in~\eqref{eq:extended_canonical_metric} for the $g$-symmetric equation~\eqref{eq:SOL-sym} in SOL-sym.

For each iteration, we seek a tangent component $\widetilde{T}(X_k)\in\T_{X_k}\St_{X_k^\top X^{}_k}(p,n)$ satisfying the enforcing condition~\cite{EisenstatEWcondition1996}
\begin{equation}
    \label{eq:inexact_condition}
    \normF{r(X_k)} \le \min\left\{\zeta_{\max}, \normF{b(X_k)}^\theta\right\} \normF{b(X_k)},
\end{equation}
Here, $b(X):=-\grad f(X)-\mathcal{A}_N(X)[N(X)]$ denotes the right-hand side of Newton equations,
and $r(X_k)$ denotes the residual of the corresponding linear system, $r(X) := \mathcal{A}_T(X)[\widetilde T(X)]-b(X)$ for SOL, or $r(X) := \Hess f(X)[\widetilde T(X)]-b(X)$ for SOL-sym. The parameters $\zeta_{\max} \in (0, 1)$ and $\theta > 0$ control the adaptive forcing sequence, guaranteeing local superlinear convergence, which becomes quadratic when $\theta=1$.

The convergence analysis mirrors that of the exact SOL, modified only by the introduction of the inner solver residual. As the following corollary establishes, under the same regularity conditions, the inexact variant achieves a local superlinear convergence rate, of order $\min\{2,1+\theta\}$.

\begin{corollary}[superlinear convergence]
\label{cor:inexact_local}
Suppose \cref{asm:smoothness,asm:regularity} hold. Let $(X_k)$ be generated by the unit-step inexact SOL iteration, given by
\[
X_{k+1}=X_k+\Lambda(X_k),
\qquad
\Lambda(X_k)=\widetilde{T}(X_k)+N(X_k),
\]
where the normal component $N(X_k)$ is given by~\eqref{eq:normal_component}, and the tangent component
$\widetilde{T}(X_k)\in \T_{X_k}\St_{X_k^\top X^{}_k}(p,n)$
satisfies the enforcing condition~\eqref{eq:inexact_condition} with parameters $\zeta_{\max}\in(0,1)$ and $\theta>0$. Then there exists $\tilde\delta>0$ such that, for all
$X_0\in B_{\tilde\delta}(X_\star)\cap \St(p,n)^\varepsilon$,
the iterates remain in $\St(p,n)^\varepsilon$ and converge to $X_\star$ with order $\min\{2,1+\theta\}$. 
\end{corollary}

\begin{proof}
Let $\Lambda_k^{\mathrm{ex}}$ denote the exact SOL update at $X_k$, i.e.,
\begin{equation*}
\mathcal A(X_k)[\Lambda_k^{\mathrm{ex}}]=-\Lambda_1(X_k).
\end{equation*}
Define the linear residual associated with $\Lambda_k$ by
\[
r_k:=\mathcal A(X_k)[\Lambda_k]+\Lambda_1(X_k).
\]
Since the normal component is computed exactly and only the tangent component is solved inexactly, it follows from~\eqref{eq:inexact_condition} that
\begin{equation}\label{eq:rk_bound}
\normF{r_k}=\normF{\mathcal A_\T(X_k)[\widetilde{T}(X_k)]-b(X_k)}
\le \xi_k\,\normF{b(X_k)},\quad
\xi_k:=\min\{\zeta_{\max},\normF{b(X_k)}^\theta\}.
\end{equation}

Proceeding as in the proof of \cref{thm:quadratically}, let
$\kappa_\star:=\|\mathcal J_{\Lambda_1}(X_\star)^{-1}\|$ and choose $\delta_A>0$ so that
$\mathcal A(X)$ is invertible on $B_{\delta_A}(X_\star)$ with
\begin{equation}\label{eq:Ainv_bd_inexact}
\|\mathcal A(X)^{-1}\|\le 2\kappa_\star,\qquad \text{for~all}\,X\in B_{\delta_A}(X_\star).
\end{equation}
Choose $\delta_F>0$ so that $B_{\delta_F}(X_\star)\subset \St(p,n)^\varepsilon$.
By \cref{asm:smoothness} and the definitions of $b$, $\mathcal A_\N$, and $N$, the mapping $b$ is locally Lipschitz continuous around $X_\star$. 
Moreover, since $X_\star$ is a stationary point and $N(X_\star)=0$, we have 
$b(X_\star)=0$. Hence, there exist $L_b>0$ and $\delta_b>0$ such that
% Since $b(\cdot)$ is continuously differentiable and $b(X_\star)=0$, there exist $L_b>0$ and $\delta_b>0$ such that
\begin{equation}\label{eq:b_Lip}
\normF{b(X)} \le L_b\normF{X-X_\star},\qquad \text{for~all}\,X\in B_{\delta_b}(X_\star).
\end{equation}

Define
$C_2:=2\kappa_\star\Big(L_\Delta+\tfrac12L_2'\Big),
C_{1+\theta}:=2\kappa_\star\,L_b^{1+\theta}$,
and pick $\tilde\delta>0$ as
$\tilde\delta
:=\min\left\{\delta_0,\delta_A,\delta_F,\delta_b,\ \zeta_{\max}^{1/\theta}/{L_b},\ 
\tilde\delta_{\mathrm{inv}}\right\}$,
where $0<\tilde\delta_{\mathrm{inv}}$ satisfies $\ C_2\tilde\delta_{\mathrm{inv}}+C_{1+\theta}\tilde\delta_{\mathrm{inv}}^\theta\le 1$.

Now fix $X_k\in B_{\tilde\delta}(X_\star)$ and let $e_k:=X_k-X_\star$.
By~\eqref{eq:b_Lip}, $\normF{b(X_k)}\le L_b\normF{e_k}\le \zeta_{\max}^{1/\theta}$, hence $\xi_k=\normF{b(X_k)}^\theta$ in
\eqref{eq:rk_bound} and thus $\normF{r_k}\le \normF{b(X_k)}^{1+\theta}$.
Subtracting the exact and inexact linear systems yields
\[
\mathcal A(X_k)[\Lambda_k-\Lambda_k^{\mathrm{ex}}]=r_k
\quad\Rightarrow\quad
\normF{\Lambda_k-\Lambda_k^{\mathrm{ex}}} \le \|\mathcal A(X_k)^{-1}\|\,\normF{r_k}.
\]
Using~\eqref{eq:Ainv_bd_inexact} and~\eqref{eq:b_Lip} gives
\begin{equation}\label{eq:dir_err}
\normF{\Lambda_k-\Lambda_k^{\mathrm{ex}}}
\le 2\kappa_\star \normF{b(X_k)}^{1+\theta}
\le C_{1+\theta}\normF{e_k}^{1+\theta}.
\end{equation}

Finally,
\[
e_{k+1}=e_k+\Lambda_k=(e_k+\Lambda_k^{\mathrm{ex}})+(\Lambda_k-\Lambda_k^{\mathrm{ex}}).
\]
By \cref{thm:quadratically} (cf.~\eqref{eq:qquad}),
$\normF{e_k+\Lambda_k^{\mathrm{ex}}} \le C_2\normF{e_k}^2$.
Combining with~\eqref{eq:dir_err} gives
\[
\normF{e_{k+1}} \le C_2\normF{e_k}^2 + C_{1+\theta}\normF{e_k}^{1+\theta}.
\]
Thus the iterates converge to $X_\star$ with order $\min\{2,1+\theta\}$.
Moreover, since $\normF{e_k}\le \tilde\delta$ and $C_2\tilde\delta + C_{1+\theta}\tilde\delta^\theta\le 1$, we obtain
$\normF{e_{k+1}}\le \normF{e_k}\le \tilde\delta$, so all iterates remain in
$B_{\tilde\delta}(X_\star)\subset \St(p,n)^\varepsilon$.
\end{proof}

\section{Numerical experiments}
\label{sec:numerical_examples}
In this section, we perform numerical experiments for solving~\eqref{eq:problem}---on the orthogonal Procrustes problem, principal component analysis, and independent component analysis---to illustrate the efficiency of the second landing methods, SOL and SOL-sym, presented in~\cref{alg:second-order_landing}. 

The compared approaches include the first-order landing method\footnote{Available at: \url{https://github.com/simonvary/landing-stiefel}.} (Landing), the inexact Riemannian Newton method (IRNM)~\cite{absilOptimizationAlgorithmsMatrix2008a,boumalIntroductionOptimizationSmooth2023b}, an exact-penalty Newton method for optimization on the Stiefel manifold (ExPen)~\cite{xiaoLiuExPen2024} with the penalty parameter set to $10$, and
a sequential quadratic programming (SQP) method\footnote{Available at: \url{https://github.com/scipy/scipy/tree/main/scipy/optimize/_trustregion_constr}.}~\cite{JNBOSQP1998}. For fair comparison, all linear systems are solved inexactly via Krylov subspace methods by the enforcing condition~\eqref{eq:inexact_condition} with $\theta=1.0$ and $\zeta_{\max}=0.1$. Specifically, BiCGSTAB\footnote{We use the SciPy implementation \texttt{scipy.sparse.linalg.bicgstab}; see \url{https://docs.scipy.org/doc/scipy/reference/generated/scipy.sparse.linalg.bicgstab.html}.} is used for solving the (non-symmetric) approximate Newton equation~\eqref{eq:computable-tangent-system} in SOL, and MINRES under $g$ metric~\eqref{eq:extended_canonical_metric} is employed for solving the ($g$-symmetric) modified Newton equation~\eqref{eq:SOL-sym} in SOL-sym. The standard MINRES is adopted to solve the Newton equations in IRNM and ExPen. Unless otherwise stated, the parameters in these approaches adopt the default settings. 

In order to reach a local regime for second-order methods, we first run the first-order landing method~\eqref{eq:FOL_iteration} until a prescribed accuracy is achieved, then initialize the second-order method at the obtained iterate. The feasible method, IRNM, is implemented on the Stiefel manifold with a standard QR-based retraction, and is initialized at the projection onto the Stiefel manifold of the point where the infeasible second-order methods are initialized. Throughout the experiments, we use the violation of first-order optimality
\begin{equation*}
    \normF{\grad f(X_k)}+
    \normF{X_k^\top X^{}_k-I_p} \leq \texttt{tol} = 10^{-12}
\end{equation*}
and the maximum iteration number $\texttt{mxit}=200$ as the stopping criterion, where the Riemannian gradient $\grad f(X_k) = 2\myskew(\nabla f(X_k)X_k^\top)X_k$ as in~\eqref{eq:Riemannian_gradient}.

All experiments are performed on a workstation with two Intel(R) Xeon(R) Processors Gold 6330 (at 2.00GHz$\times$28, 42M Cache) and 512GB of RAM running Python (Release 3.12.1) under Ubuntu 22.04.3. The code that produced the results is available at \url{https://github.com/XinhuiXiong/SOLanding}.

\subsection{Orthogonal Procrustes problem}
To investigate the effect of the tangent component in the second-order landing framework~\eqref{eq:second_order_landing_framework}, we compare the SOL and SOL-sym methods. Consider the orthogonal Procrustes problem~\cite{schonemannOrthogonalProcrustes1966}
\[
    \min_{X\in \St(d,d)} \frac{1}{2n}\|AX-B\|_{\mathrm F}^2,
\]
where $A,B\in\mathbb R^{n\times d}$ are given matrices. 
We generate synthetic instances with ground-truth solutions. We set the dimension $(n, d) = (10000, 1000)$, generate $A$ as a standard Gaussian matrix, and form $B = A X_{\mathrm{true}} + \sigma \varXi$, where $X_{\mathrm{true}}\in 
\St(d,d)$ is a random ground-truth matrix, $\varXi$ is a standard Gaussian noise matrix, and the noise level is $\sigma=0.02$. An initial point $X_0\in\mathbb R^{n\times d}$ is generated to satisfy $\normF{\grad f(X_0)}\leq10^{-2}$.

The evolution of the norm of the Riemannian gradient and feasibility violation is reported in~\cref{fig:procrustes_all}. We compare SOL and SOL-sym with the first-order landing method, and observe that SOL and SOL-sym exhibit fast local convergence. The empirical results validate the effectiveness of SOL and SOL-sym in accelerating local convergence. Note that SOL outperforms SOL-sym thanks to its lower per-iteration cost.

\begin{figure}[htbp]
    \centering
    \begin{subfigure}[t]{0.49\textwidth}
        \centering
        \includegraphics[width=\textwidth]{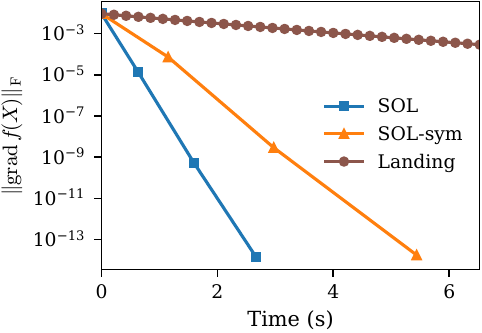}
        \label{fig:procrustes_tangent_time}
    \end{subfigure}
    \hfill
    \begin{subfigure}[t]{0.49\textwidth}
        \centering
        \includegraphics[width=\textwidth]{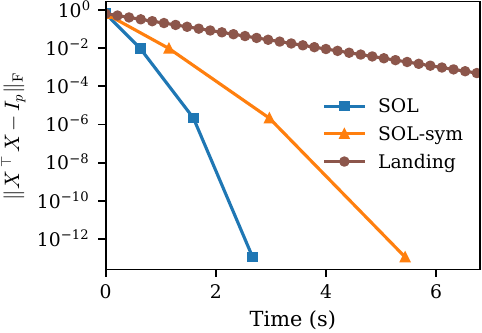}
        \label{fig:procrustes_ortho_time}
    \end{subfigure}
    \caption{Numerical comparison of different methods on the orthogonal Procrustes problem. {Left:} gradient norm; {Right:} feasibility violation.}
    \label{fig:procrustes_all}
\end{figure}

\subsection{Principal component analysis}
We compare the proposed second-order landing methods, SOL and SOL-sym, with existing methods on the principal component analysis (PCA) problem~\cite{PearsonPCA1901}, which is formulated as
\[
\min_{X\in\St(p,n)} -\frac{1}{N}\tr (X^\top A^\top A X),
\]
where $A\in\R^{N\times n}$ is a synthetically generated data matrix, with $N=30000$ samples, $n=10000$ features, and $p=500$ target principal components.
The rows of $A$ are independently sampled from the normal distribution $\mathcal{N}(0,UU^\top+\sigma I_n)$, where $U\in\R^{n \times p}$ is sampled uniformly from $\St(p,n)$, and $\sigma=0.1$. 
An initial point $X_0\in\mathbb R^{n\times d}$ is generated to satisfy $\normF{\grad f(X_0)}\leq10^{-2}$. 
To avoid oversolving the subproblems, we set the parameters in the enforcing condition~\eqref{eq:inexact_condition} as $\theta=0.5$ and $\zeta_{\max}=0.1$.

\begin{figure}[htbp]
    \centering
        \begin{subfigure}{0.49\textwidth}
        \centering
        \includegraphics[width=\linewidth]{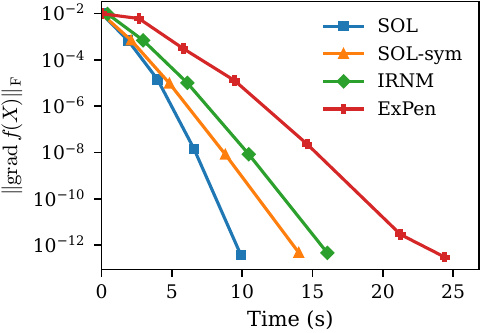}
        \label{fig:pca_tangent_time}
    \end{subfigure}
        \hfill 
        \begin{subfigure}{0.49\textwidth}
        \centering
        \includegraphics[width=\linewidth]{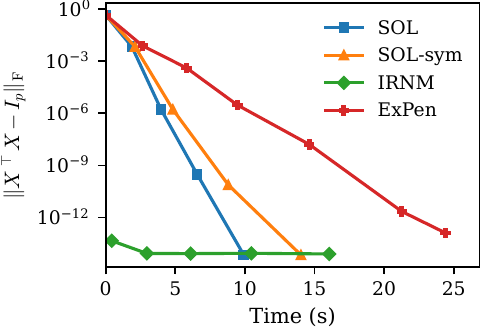}
        \label{fig:pca_ortho_time}
    \end{subfigure}

    \caption{Numerical comparison of different methods on principal component analysis. {Left:} gradient norm; {Right:} feasibility violation.}
    \label{fig:gevp_convergence_all}
\end{figure}

\Cref{fig:gevp_convergence_all} shows that SOL and SOL-sym outperform the other methods. While both SOL and SOL-sym demonstrate fast local convergence, SOL achieves the best practical performance due to a tractable Newton equation~\eqref{eq:computable-tangent-system}. Regarding feasibility, all three infeasible methods eventually drive the feasibility violation to the same level as the feasible method IRNM.

\subsection{Independent component analysis}
\label{subsec:exp_ica}
We further evaluate our methods on independent component analysis (ICA), a fundamental task in blind source separation. The goal is to find an orthogonal unmixing matrix $X$ that maximizes the non-Gaussianity of the projected data. Following the classical log-cosh contrast function from FastICA~\cite{hyvarinenFastRobustFixedPoint1999}, we formulate the problem as
\[
\min_{X\in\mathrm{St}(d,d)} f(X) := -\frac{1}{N} \sum_{i=1}^N \sum_{j=1}^p \log\big(\cosh((W X)_{i,j})\big),
\]
where $W \in \mathbb R^{N\times d}$ is the whitened data matrix, $N$ is the number of samples, and $(WX)_{i,j}$ denotes the $(i,j)$-th entry of the matrix $WX$.

The experiment is conducted on real electroencephalography (EEG) data drawn from the MNE-Python sample dataset.\footnote{The real-data experiment is based on the public MNE-Python sample dataset, available at \url{https://mne.tools/stable/documentation/datasets.html}, specifically the raw file \texttt{MEG/sample/sample\_audvis\_raw.fif}.}
The raw EEG signals from the first $60$ seconds are high-pass filtered at $1$\,Hz, resampled to $100$\,Hz, centered, and subsequently whitened using the truncated SVD to obtain the input matrix $W \in \mathbb{R}^{N \times d}$. With the number of temporal samples $N=6000$ and $d=60$ retained principal components, the resulting optimization problem is defined on the manifold $\St(60,60)$. In this experiment, we broaden our comparison by including SQP alongside the previously evaluated solvers. Consistent with previous tests, all compared second-order methods are initialized from a warm-start point $X_0$ generated by the first-order landing method satisfying $\normF{\grad f(X_0)} \leq 10^{-3}$. To obtain better solutions, we set the stopping tolerance to $\texttt{tol}=10^{-13}$.

\begin{figure}[htbp]
    \centering
    \begin{subfigure}[t]{0.49\textwidth}
        \centering
        \includegraphics[width=\linewidth]{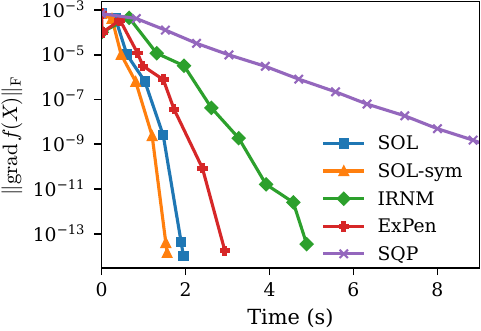}
        \label{fig:ica_tangent_time}
    \end{subfigure}
    \hfill 
    \begin{subfigure}[t]{0.49\textwidth}
        \centering
        \includegraphics[width=\linewidth]{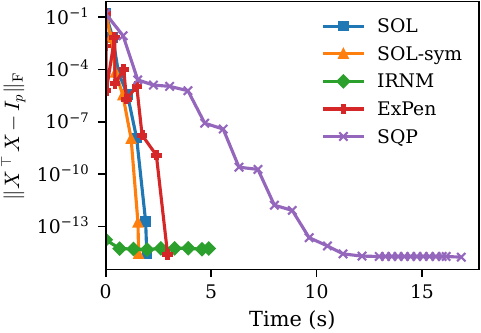}
        \label{fig:ica_ortho_time}
    \end{subfigure}

    \caption{Numerical comparison of different methods on independent component analysis with real EEG data. {Left:} gradient norm; {Right:} feasibility violation.}
    \label{fig:ica_convergence_all}
\end{figure}

\Cref{fig:ica_convergence_all} reports the numerical results on real-data ICA. Among the compared methods, the proposed SOL and SOL-sym methods exhibit the fastest convergence to a high-accuracy solution and perform better than the other methods. Unlike in high-dimensional scenarios, the performance of SOL-sym is comparable to SOL in this test since the computational cost of the Riemannian Hessian in~\eqref{eq:SOL-sym} is no longer dominant. The Riemannian method IRNM is able to preserve the feasibility of iterates, and the infeasible methods reduce the violation of feasibility to the same level, but without invoking retractions, thereby reducing computational time. Additionally, SQP is designed for general equality-constrained optimization and does not explicitly exploit the geometry of the Stiefel manifold, which leads to a higher computational cost.

\section{Conclusions and perspectives}
\label{sec:conclusion}
We have proposed a second-order landing framework for optimization on the Stiefel manifold, in which each update is decomposed into normal and tangent components. The normal component is chosen as the order-1 Newton-Schulz update. The main lesson is that such a normal step, added to a Riemannian Newton update, does not produce quadratic convergence; the Riemannian Newton equation must be modified to account for the normal-induced perturbation of the Riemannian gradient. This viewpoint leads to two second-order landing methods: SOL-sym, which retains the full Riemannian Hessian, and SOL, which further reduces the computational cost. Numerical results demonstrate that the proposed retraction-free methods not only enjoy fast local convergence but also achieve highly accurate orthogonality, which is often lacking in retraction-free first-order methods. In practice, we recommend SOL-sym for moderate-scale problems to retain full second-order information, and SOL for large-scale problems where computational efficiency is more critical.

This work opens several avenues for future research. An interesting direction is to globalize second-order landing methods. Moreover, extending the proposed framework to generalized Stiefel manifolds, as well as developing stochastic and preconditioned variants, is highly desirable for modern machine learning applications.

\bibliographystyle{siamplain}
\bibliography{references}
\end{document}